\documentclass{article}
\usepackage{amsmath, amsthm, amssymb}
\usepackage[margin=3.5cm]{geometry}
\usepackage{hyperref} %
\usepackage{verbatim}
\usepackage{url}
\usepackage{yfonts}
\usepackage{youngtab}
\usepackage{shuffle}

\usepackage{tikz, xcolor, mathrsfs, multicol}
\usetikzlibrary{shapes.geometric, arrows}

\usepackage{enumerate}

\newtheorem{thm}{Theorem}[subsection]
\newtheorem{prop}[thm]{Proposition}
\newtheorem{lem}[thm]{Lemma}
\newtheorem{cor}[thm]{Corollary}

\newtheorem{mydef}[thm]{Definition}
\theoremstyle{remark}
\newtheorem{rem}[thm]{Remark}

\newtheorem{exm}[thm]{Example}

\title{Connecting $k$-Naples parking functions and obstructed parking functions via involutions}
\author{Roger Tian\footnote{\href{mailto:rgtian@ucdavis.edu}{rgtian@ucdavis.edu} or \href{mailto:htrland@gmail.com}{htrland@gmail.com}}}
\date{\today}

\begin{document}
\maketitle

\begin{abstract}
Parking functions were classically defined for $n$ cars attempting to park on a one-way street with $n$ parking spots, where cars only drive forward. Subsequently, parking functions have been generalized in various ways, including allowing cars the option of driving backward. The set $PF_{n,k}$ of $k$-Naples parking functions have cars who can drive backward a maximum of $k$ steps before driving forward. A recursive formula for $|PF_{n,k}|$ has been obtained, though deriving a closed formula for $|PF_{n,k}|$ appears difficult. In addition, an important subset $B_{n,k}$ of $PF_{n,k}$, called the contained $k$-Naples parking functions, has been shown, with a non-bijective proof, to have the same cardinality as that of the set $PF_n$ of classical parking functions, independent of $k$. 

In this paper, we study $k$-Naples parking functions in the more general context of $m$ cars and $n$ parking spots, for any $m \leq n$. We use various parking function involutions to establish a bijection between the contained $k$-Naples parking functions and the classical parking functions, from which it can be deduced that the two sets have the same number of ties. Then we extend this bijection to inject the set of $k$-Naples parking functions into a certain set of obstructed parking functions, providing an upper bound for the cardinality of the former set.
\end{abstract}

\section{Introduction}
Parking functions are combinatorial objects introduced by Konheim and Weiss \cite{KonWei} in their study of hashing. For the purposes of this paper, a parking function can be defined as follows: Consider $m$ cars $c_1, c_2, \ldots, c_m$ trying to park, in that order, randomly on a one-way street with $n$ parking spots, labeled $1, 2, \ldots, n$. For each $j \in [m]$, car $c_j$ has a preferred parking spot $a_j \in [n]$ and will park there if the spot is unoccupied. If $a_j$ is occupied, then car $c_j$ tries to park at the next parking spot, and this process continues until $c_j$ reaches an unoccupied parking spot and parks there, or it terminates due to lack of available parking spots. This street can be modeled by a directed path $I$ with vertices labeled $1, 2, \ldots, n$, and this sequence $s \in [n]^m$ of car preferences is called a \textbf{classical $(m,n)$-parking function} if all $m$ cars manage to park through this process. Denote by $PF(m,n)$ the set of classical $(m,n)$-parking functions. For $m \leq n$, it has been shown that $|PF(m,n)| = (n-m+1)(n+1)^{m-1}$ \cite{CaJoSch}. 

Classical parking functions have appeared in the study of topics such as tree enumeration \cite{GesSeo}, noncrossing partitions \cite{Stan}, and acyclic functions \cite{Yan15}. They have been generalized in various directions, including parking functions on digraphs \cite{KingYan} and Naples parking functions \cite{Baum}. In this paper, we will explore a generalization of the latter called the $k$-Naples parking functions \cite{Christ}.

Fix integer $0 \leq k \leq n-1$. Consider the following modification to the parking rule for the aforementioned $m$ cars, which will allow the cars to drive first backward (left), then forward (right) before parking. If car $c_j$ finds its preferred vertex $a_j$ occupied, then it travels to the nearest available vertex in the $k$ spots before $a_j$ and parks there if such vertex exists. If all $k$ vertices before $a_j$ are occupied, then $c_j$ drives to the nearest available vertex after $a_j$ and parks there if such vertex exists. The $m$-tuple $(a_1,a_2,\ldots,a_m)$ of car preferences is called a \textbf{$k$-Naples $(m,n)$-parking function} if all $m$ cars are able to park via this process. Denote by $PF(m,n;k)$ the set of $k$-Naples $(m,n)$-parking functions, with $PF(m,n;0) = PF(m,n)$.

There is an important subset $B(m,n;k) \subset PF(m,n;k)$, called the \textbf{contained $k$-Naples parking functions}, consisting of those elements of $PF(m,n;k)$ such that not all vertices before $a_j$ are occupied whenever $a_j \leq k$, for all $j \in [m]$. In other words, no car of $B(m,n;k)$ has to ``exit'' the parking lot $I$ in order to finish its backward search through $k$ vertices for an available spot.

In the case $m = n$, denote $PF_{n,k} := PF(n,n;k)$ and $B_{n,k} := B(n,n;k)$, and $|PF_{n,k}|$ has the following recursive formula \cite{Christ}:
\begin{equation}
\label{kN recur}
|PF_{n+1,k}| = \sum_{i=0}^n{{n \choose i}\min(i+1+k,n+1)|PF_{i,k}||B_{n-i,k}|}.
\end{equation}

It was also proved in \cite{Christ} that $$|B_{n,k}| = (n+1)^{n-1}$$ independent of $k$, using a modified version of Pollack's technique \cite{Poll}. Since this proof was not bijective, the authors posed the open problem of finding a bijection between $B_{n,k}$ and the set $PF_n := PF_{n,0}$ of classical $(n,n)$-parking functions that ``preserves'' in some manner certain parking function statistics such as ascents, descents, and ties. 

In addition, as the authors pointed out, it would be very difficult to derive a closed formula for $|PF_{n,k}|$ using the recursive formula (\ref{kN recur}), given that even special cases of (\ref{kN recur}) include quantities for which no closed formulae are known. Therefore, results on the size of $|PF_{n,k}|$ would be desirable.

Certain parking functions in which some vertices have been obstructed and hence unavailable for parking have been studied in \cite{Ehren}, in the form of cars parking after a fixed trailer. We will tackle the previously raised questions through our effort to connect $k$-Naples parking functions with obstructed parking functions.

In this paper, we will study $k$-Naples parking functions using various involutions on these objects. For general $m \leq n$, we introduce such an involution in Section \ref{cla invo} and apply it to construct a bijection between $B(m,n;k)$ and $PF(m,n)$ in Section \ref{conta bij}. We then use this bijection to deduce that $B(m,n;k)$ and $PF(m,n)$ have the same number of ties, in Section \ref{conta ties}. Lastly, in Section \ref{obstr bound}, we extend this bijection to an injection of $PF(m,n;k)$ into $LPF(m,n+k;k)$, the set of parking functions with $m$ cars and $n+k$ vertices whose first $k$ vertices are obstructed. This will then allow us to obtain the upper bound 
\begin{equation}
\label{gen ub}
|PF(m,n;k)| \leq |LPF(m,n+k;k)|,
\end{equation} 
with equality only if $k=0$.
Finally, in the case $m = n$, we combine (\ref{gen ub}) with a special case of the main result of \cite{Ehren} to obtain 
\begin{equation}
\label{spc ub}
|PF_{n,k}| \leq (k+1)(k+n+1)^{n-1},
\end{equation}
with equality only if $k=0$.

\section{Acknowledgments}
The author would like to thank Pamela Harris for helpful conversations via email.

\section{Preliminaries}
For any two tuples $u, v$ we denote by $u \oplus v$ the concatenation of $u$, $v$ and we denote by $\mathrm{len}(v)$ the number of entries of $v$. Let $m \leq n \in \mathbb{N}$ and let $PF(m,n)$ denote the set of parking functions with $m$ cars and $n$ parking spots--called the classical $(m,n)$-parking functions. We regard the parking lot as a directed path $I$ with vertices $1, 2, \ldots, n$, and we will denote each path on $I$ as a tuple of its vertices. For vertices $a<b \in [n]$ we denote by $\mathrm{path}(a,b)$ the path on $I$ with $a$ and $b$ as left and right endpoints respectively. Now we give the definitions of the relevant parking functions introduced in \cite{Christ}.

\begin{mydef}
Fix integers $1 \leq m \leq n$ and $0 \leq k \leq n-1$. Consider $m$ cars trying to park on a directed path $I$ with $n$ vertices one after another, with car $c_j$ preferring vertex $a_j$. If car $c_j$ finds $a_j$ unoccupied, it parks there. If $a_j$ is occupied, then $c_j$ drives backward from $a_j$ one vertex at a time to check if any vertices in the set $A_{j,k} := \{a_j-1,a_j-2,\ldots,a_j-k\} \cap [n]$ are available. If $A_{j,k}$ has unoccupied vertices, then $c_j$ parks in the available vertex $a_l \in A_{j,k}$ closest to $a_j$. If all vertices in $A_{j,k}$ are occupied, then $c_j$ drives forward and parks in the first unoccupied vertex after $a_j$ if it exists. If all cars are able to park under this rule, then we call the tuple $f = (a_1,a_2,\ldots,a_m)$ of car preferences a \textbf{$k$-Naples $(m,n)$-parking function}. Denote by $PF(m,n;k)$ the set of all $k$-Naples $(m,n)$-parking functions.
\end{mydef}

\begin{mydef}
Fix integers $1 \leq m \leq n$ and $0 \leq k \leq n-1$. The set of \textbf{contained $k$-Naples $(m,n)$-parking functions} $B(m,n;k)$ is the subset of $PF(m,n;k)$ such that if car $c_i$ has preference $a_i \leq k$, then some vertex in $\{1,2,\ldots,a_i\}$ has not been occupied by any car in $\{c_1,c_2,\ldots,c_{i-1}\}$.
\end{mydef}

\begin{exm}
Consider the following directed path $I$. 

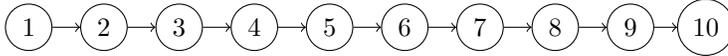
\begin{tikzpicture}
\node[shape=circle,draw=black] (1) at (0,0) {1};
\node[shape=circle,draw=black] (2) at (1,0) {2};
\node[shape=circle,draw=black] (3) at (2,0) {3};
\node[shape=circle,draw=black] (4) at (3,0) {4};
\node[shape=circle,draw=black] (5) at (4,0) {5};
\node[shape=circle,draw=black] (6) at (5,0) {6};
\node[shape=circle,draw=black] (7) at (6,0) {7};
\node[shape=circle,draw=black] (8) at (7,0) {8};
\node[shape=circle,draw=black] (9) at (8,0) {9};
\node[shape=circle,draw=black] (10) at (9,0) {10};

\path [->] (1) edge (2);
\path [->] (2) edge (3);
\path [->] (3) edge (4);
\path [->] (4) edge (5);
\path [->] (5) edge (6);
\path [->] (6) edge (7);
\path [->] (7) edge (8);
\path [->] (8) edge (9);
\path [->] (9) edge (10);
\end{tikzpicture}

We have $(2,4,8,9,2,8,9,9,9,3) \in B(10,10;4)$.

We have $(2,4,6,9,2,6,2,9,3) \in PF(9,10;4)-B(9,10;4)$ because cars 7 and 9 would have to exit the parking lot in order to finish their backward search through $k$ vertices.

We have $(2,3,6,9,9,6,8,9,9) \notin PF(9,10;4)$ because car 9 cannot park.
\end{exm}

\begin{mydef}
Let $f = (a_1,a_2,\ldots,a_m) \in PF(m,n;k)$. For all $1 \leq j \leq n-1$, we say that 
\begin{enumerate}
\item $(j,j+1)$ is an \textbf{ascent} if $a_j < a_{j+1}$
\item $(j,j+1)$ is a \textbf{descent} if $a_j > a_{j+1}$
\item $(j,j+1)$ is a \textbf{tie} if $a_j = a_{j+1}$.
\end{enumerate}
\end{mydef}

For the parking functions with obstructions we consider here, the parking rule is the same as that of the classical parking functions, but just with certain additional vertices unavailable for parking though they can still be preferred. 

\begin{mydef}
Fix integers $1 \leq m \leq n$ and $0 \leq k \leq n-1$. Consider $m$ cars trying to park on a directed path $I$ with $n+k$ vertices one after another, where certain predetermined $k$ consecutive vertices are occupied by obstructions and hence unavailable for parking. Each car $c_j$ will park at its preferred vertex $a_j$ if the spot is unoccupied. If $a_j$ is occupied either by an obstruction or by an earlier car, then $c_j$ will keep driving forward from $a_j$ until it finds an available vertex and parks there, or the process terminates due to lack of available vertices. The sequence $f = (a_1,a_2,\ldots,a_m)$ of car preferences is called a \textbf{$k$-consecutive obstructed $(m,n)$-parking function} if all cars successfully park through this process. Denote by $OPF(m,n+k;k)$ the set of $k$-consecutive obstructed $(m,n)$-parking functions.
\end{mydef}

\begin{mydef}
Fix integers $1 \leq m \leq n$ and $0 \leq k \leq n-1$. The set $LPF(m,n+k;k)$ of \textbf{$k$-left obstructed $(m,n)$-parking functions} is the subset of $OPF(m,n+k;k)$ such that the first $k$ vertices are obstructed.
\end{mydef}

\begin{exm}
Consider the following directed path $I$, where the obstructed vertices are labeled with a prime ('). \\ 
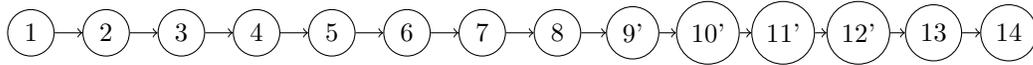
\begin{tikzpicture}
\node[shape=circle,draw=black] (1) at (0,0) {1};
\node[shape=circle,draw=black] (2) at (1,0) {2};
\node[shape=circle,draw=black] (3) at (2,0) {3};
\node[shape=circle,draw=black] (4) at (3,0) {4};
\node[shape=circle,draw=black] (5) at (4,0) {5};
\node[shape=circle,draw=black] (6) at (5,0) {6};
\node[shape=circle,draw=black] (7) at (6,0) {7};
\node[shape=circle,draw=black] (8) at (7,0) {8};
\node[shape=circle,draw=black] (9) at (8,0) {9'};
\node[shape=circle,draw=black] (10) at (9,0) {10'};
\node[shape=circle,draw=black] (11) at (10,0) {11'};
\node[shape=circle,draw=black] (12) at (11,0) {12'};
\node[shape=circle,draw=black] (13) at (12,0) {13};
\node[shape=circle,draw=black] (14) at (13,0) {14};

\path [->] (1) edge (2);
\path [->] (2) edge (3);
\path [->] (3) edge (4);
\path [->] (4) edge (5);
\path [->] (5) edge (6);
\path [->] (6) edge (7);
\path [->] (7) edge (8);
\path [->] (8) edge (9);
\path [->] (9) edge (10);
\path [->] (10) edge (11);
\path [->] (11) edge (12);
\path [->] (12) edge (13);
\path [->] (13) edge (14);
\end{tikzpicture}

We have $(1,1,2,3,5,5,5,10,12,8) \in OPF(10,14;4)$ but we have $(3,3,1,1,6,7,8,9,8,8) \notin OPF(10,14;4)$.
\end{exm}

\begin{exm}
Consider the following directed path $I$, where the obstructed vertices are labeled with a prime ('). \\ 
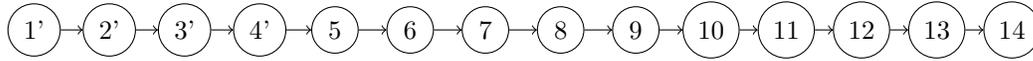
\begin{tikzpicture}
\node[shape=circle,draw=black] (1) at (0,0) {1'};
\node[shape=circle,draw=black] (2) at (1,0) {2'};
\node[shape=circle,draw=black] (3) at (2,0) {3'};
\node[shape=circle,draw=black] (4) at (3,0) {4'};
\node[shape=circle,draw=black] (5) at (4,0) {5};
\node[shape=circle,draw=black] (6) at (5,0) {6};
\node[shape=circle,draw=black] (7) at (6,0) {7};
\node[shape=circle,draw=black] (8) at (7,0) {8};
\node[shape=circle,draw=black] (9) at (8,0) {9};
\node[shape=circle,draw=black] (10) at (9,0) {10};
\node[shape=circle,draw=black] (11) at (10,0) {11};
\node[shape=circle,draw=black] (12) at (11,0) {12};
\node[shape=circle,draw=black] (13) at (12,0) {13};
\node[shape=circle,draw=black] (14) at (13,0) {14};

\path [->] (1) edge (2);
\path [->] (2) edge (3);
\path [->] (3) edge (4);
\path [->] (4) edge (5);
\path [->] (5) edge (6);
\path [->] (6) edge (7);
\path [->] (7) edge (8);
\path [->] (8) edge (9);
\path [->] (9) edge (10);
\path [->] (10) edge (11);
\path [->] (11) edge (12);
\path [->] (12) edge (13);
\path [->] (13) edge (14);
\end{tikzpicture}

We have $(7,8,8,2,3,12,11,11,14,4) \in LPF(10,14;4)$ but we have $(4,4,8,8,9,13,12,13,12)$ $\notin LPF(9,14;4)$.
\end{exm}

\section{Contained $k$-Naples parking functions}
We will construct a bijection $\Xi_{m,n;k}$ between $B(m,n;k)$ and $PF(m,n)$ that ``preserves'' the distance traversed by each car before parking, taking into consideration the cars that park at or before their preferred spots and the cars that park after. Here, the cars of $B(m,n;k)$ that park backward will correspond to the cars of $PF(m,n)$ traversing no more than $k$ vertices, while the cars of $B(m,n;k)$ that park forward will correspond to the cars of $PF(m,n)$ traversing more than $k$ vertices. To deal with these two types of cars, it is useful to first construct an involution $\Phi_{m,n}$ on $PF(m,n)$ that, in some sense, ``inverts'' the configuration of the parked cars. Each $f \in B(m,n;k)$ will be broken up into a sequence of sub-tuples, consisting of the preferences of backward parkers, the forward parkers, the backward parkers, and so on. $\Xi_{m,n;k}(f)$ will be constructed one sub-tuple at a time, with $\Phi_{m,n}$ being applied each time we move onto the next sub-tuple.

\subsection{Involution on classical $(m,n)$-parking functions}
\label{cla invo}
We introduce an involution $\Phi_{m,n}$ on $PF(m,n)$ that ``reflects'' the parking configuration; call $\Phi_{m,n}$ the \textbf{parking reflection on $PF(m,n)$}. Let $f \in PF(m,n)$. For each car $c \in [m]$, let $\mathrm{tpath}(c)$ denote the path car $c$ must traverse before parking; we call $\mathrm{tpath}(c)$ the \textbf{traverse path of $c$}.

For vertices $i, j \in [n]$, we write $i \sim j$ if there exist cars $c_1, c_2, \ldots, c_k$ such that $\mathrm{tpath}(c_1)$ contains $i$, $\mathrm{tpath}(c_k)$ contains $j$, and $\mathrm{tpath}(c_l) \cap \mathrm{tpath}(c_{l+1}) \neq \emptyset$. $\sim$ is a relation that yields ``path components'' that we will call \textbf{parking components}; every two vertices in a parking component are connected by a sequence of overlapping traverse paths. We will denote a parking component as a path, which is the union of all the relevant traverse paths.

Each parking function consists of a collection of parking components. Notice that, for any parking component $L$, the cars preferring $L$ never leave $L$, and the cars preferring $I-L$ never enter $L$. To define $\Phi_{m,n}$, the idea is to reflect each parking component about the center of $I$ while preserving the order of preferences of all cars that park along the parking component. 

\begin{mydef}[Parking reflection $\Phi_{m,n}$]
Given a parking function $f \in PF(m,n)$, define $\Phi_{m,n}(f)$ as follows. Let $L = (i,i+1,i+2,\ldots,i+j)$ be a parking component of $f$, with its reflection $(n-i+1,n-i,n-i-1,\ldots,n-i-j+1)$ which we rewrite as the path $(n-i-j+1,n-i-j+2,\ldots,n-i+1)$. For any car $c$ preferring $i+a$ in $f$ for $0 \leq a \leq j$, $c$ will prefer $n-i-j+a+1$ in $\Phi_{m,n}(f)$.
\end{mydef}
Thus defined, $\Phi_{m,n}$ is clearly an involution on $PF(m,n)$, which also preserves the traverse path length of each car.

\begin{exm}
Consider the following directed path $I$. 

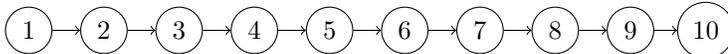
\begin{tikzpicture}
\node[shape=circle,draw=black] (1) at (0,0) {1};
\node[shape=circle,draw=black] (2) at (1,0) {2};
\node[shape=circle,draw=black] (3) at (2,0) {3};
\node[shape=circle,draw=black] (4) at (3,0) {4};
\node[shape=circle,draw=black] (5) at (4,0) {5};
\node[shape=circle,draw=black] (6) at (5,0) {6};
\node[shape=circle,draw=black] (7) at (6,0) {7};
\node[shape=circle,draw=black] (8) at (7,0) {8};
\node[shape=circle,draw=black] (9) at (8,0) {9};
\node[shape=circle,draw=black] (10) at (9,0) {10};

\path [->] (1) edge (2);
\path [->] (2) edge (3);
\path [->] (3) edge (4);
\path [->] (4) edge (5);
\path [->] (5) edge (6);
\path [->] (6) edge (7);
\path [->] (7) edge (8);
\path [->] (8) edge (9);
\path [->] (9) edge (10);
\end{tikzpicture}

For the parking function $(4,9,6,8,1,8,1,2)$, the traverse path of car 6 is $\mathrm{tpath}(6) = (8,9,10)$. Since the parking components are $(1,2,3)$, $(4)$, $(6)$, and $(8,9,10)$, we have $\Phi_{m,n}(4,9,6,8,1,8,1,2)$ $= (7,2,5,1,8,1,8,9)$. 

For the parking function $(2,3,4,2,4,7,7,8)$, the traverse path of car 4 is $\mathrm{tpath}(4) = (2,3,4,5)$. Since the parking components are $(2,3,4,5,6)$ and $(7,8,9)$, we have $\Phi_{m,n}(2,3,4,2,4,7,7,8) = (5,6,7,5,7,2,2,3)$.
\end{exm}

\subsection{Bijection between classical parking functions and contained $k$-Naples parking functions}
\label{conta bij}
Fix $0 \leq k \leq n-1$ and let $B(m,n;k)$ denote the set of contained $k$-Naples parking functions with $m$ cars. We will utilize the involutions $\Phi_{i,n}$ for various $i \in [m]$ to establish a bijection $\Xi_{m,n;k}$ between $B(m,n;k)$ and $PF(m,n)$; constructing an element of $PF(m,n)$ from an element of $B(m,n;k)$ will involve multiple ``reflections'' of parking configurations, depending on whether cars park backward or forward.

Let $f \in B(m,n;k)$. Note that cars in $f$ never start by parking forward. We partition the $m$-tuple $f$ into sub-tuples $f = f_1 \oplus f_2 \oplus \cdots \oplus f_d$ following their order in $f$ so that: 
\begin{enumerate}
\item \label{odd} $f_i$ consists of cars parking at or before their preferred spots for $i$ odd.
\item \label{even} $f_i$ consists of cars parking after their preferred spots for $i$ even.
\end{enumerate}
We will henceforth call $f = f_1 \oplus f_2 \oplus \cdots \oplus f_d$ the \textbf{$k$-decomposition of $f$}. The cars in Case \ref{odd} will correspond to the cars with traverse paths of length not exceeding $k$ in $\Xi_{m,n;k}(f)$, while the cars in Case \ref{even} will correspond to the cars with traverse paths of length exceeding $k$ in $\Xi_{m,n;k}(f)$. \\

We construct $\Xi_{m,n;k}(f)$ in stages by specifying the new preferences of the cars for each of $f_1, f_2, \ldots$ in that order; denote by $\Xi_{m,n;k}(f)_i$ the parking function obtained after $f_i$. See Remark \ref{parkrem} for a rough picture of this procedure.
\begin{enumerate}
\item For each car $c_1$ with preference $p_1 \in [n]$ in $f_1$, $c_1$ will have new preference $n+1-p_1$. This yields a classical $(\mathrm{len}(f_1),n)$-parking function $\Xi_{m,n;k}(f)_1$. 
\item For each car $c_2$ with preference $p_2$ in $f_2$, $c_2$ will have new preference $p_2-k$. This yields a $\mathrm{len}(f_2)$-tuple $\nu_2$ of preferences, and we define $\Xi_{m,n;k}(f)_2 = \Phi_{\mathrm{len}(f_1),n}(\Xi_{m,n;k}(f)_1) \oplus \nu_2$.
\item For each car $c_3$ with preference $p_3$ in $f_3$, $c_3$ will have new preference $n+1-p_3$. This yields a $\mathrm{len}(f_3)$-tuple $\nu_3$ of preferences, and we define $\Xi_{m,n;k}(f)_3 = \Phi_{\mathrm{len}(f_1)+\mathrm{len}(f_2),n}(\Xi_{m,n;k}(f)_2) \oplus \nu_3$. 
\item Continue this process as follows.
\begin{enumerate}
\item Suppose $i$ is odd. For each car $c_i$ with preference $p_i$ in $f_i$, $c_i$ will have new preference $n+1-p_i$. This yields a $\mathrm{len}(f_i)$-tuple $\nu_i$ of preferences, and we define $$\Xi_{m,n;k}(f)_i = \Phi_{\sum_{u=1}^{i-1}{\mathrm{len(f_u)}},n}(\Xi_{m,n;k}(f)_{i-1}) \oplus \nu_i.$$
\item Suppose $i$ is even. For each car $c_i$ with preference $p_i$ in $f_i$, $c_i$ will have new preference $p_i-k$. This yields a $\mathrm{len}(f_i)$-tuple $\nu_i$ of preferences, and we define $$\Xi_{m,n;k}(f)_i = \Phi_{\sum_{u=1}^{i-1}{\mathrm{len(f_u)}},n}(\Xi_{m,n;k}(f)_{i-1}) \oplus \nu_i.$$
\end{enumerate}
\item Finally, we define $\Xi_{m,n;k}(f) := \Xi_{m,n;k}(f)_d$.
\end{enumerate}
\begin{rem}
\label{parkrem}
Roughly speaking, $\Xi_{m,n;k}(f)$ is constructed in stages $\Xi_{m,n;k}(f)_1$, $\Xi_{m,n;k}(f)_2$, $\ldots$, $\Xi_{m,n;k}(f)_d = \Xi_{m,n;k}(f)$, where the $\Xi_{m,n;k}(f)_i$ are as follows: 
\begin{enumerate}
\item The vertices occupied by cars in $\Xi_{m,n;k}(f)_1$ are the reflection of those in $f_1$.
\item The vertices occupied by cars in $\Xi_{m,n;k}(f)_2$ coincide with those in $f_1 \oplus f_2$.
\item In general, the vertices occupied by cars in $\Xi_{m,n;k}(f)_i$
\begin{enumerate}
\item coincide with those in $f_1 \oplus f_2 \oplus \ldots \oplus f_i$ for $i$ even
\item are the reflection of those in $f_1 \oplus f_2 \oplus \ldots \oplus f_i$ for $i$ odd
\end{enumerate}
\end{enumerate}
\end{rem}

\begin{exm}
Consider the following directed path $I$. 

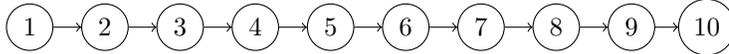
\begin{tikzpicture}
\node[shape=circle,draw=black] (1) at (0,0) {1};
\node[shape=circle,draw=black] (2) at (1,0) {2};
\node[shape=circle,draw=black] (3) at (2,0) {3};
\node[shape=circle,draw=black] (4) at (3,0) {4};
\node[shape=circle,draw=black] (5) at (4,0) {5};
\node[shape=circle,draw=black] (6) at (5,0) {6};
\node[shape=circle,draw=black] (7) at (6,0) {7};
\node[shape=circle,draw=black] (8) at (7,0) {8};
\node[shape=circle,draw=black] (9) at (8,0) {9};
\node[shape=circle,draw=black] (10) at (9,0) {10};

\path [->] (1) edge (2);
\path [->] (2) edge (3);
\path [->] (3) edge (4);
\path [->] (4) edge (5);
\path [->] (5) edge (6);
\path [->] (6) edge (7);
\path [->] (7) edge (8);
\path [->] (8) edge (9);
\path [->] (9) edge (10);
\end{tikzpicture}

For $f = (6,6,5,4,5,6,7,7) = (6,6,5,4,5) \oplus (6,7,7) \in B(8,10;4)$, we have $\Xi_{8,10;4}(f)_1 = (5,5,6,7,6)$ and $\Xi_{8,10;4}(f) = \Xi_{8,10;4}(f)_2 = \Phi_{5,10}(5,5,6,7,6) \oplus (2,3,3) = (2,2,3,4,3,2,3,3)$.

For $f = (5,5,4,4,3,5,10,6,10) = (5,5,4,4,3) \oplus (5) \oplus (10) \oplus (6) \oplus (10) \in B(9,10;4)$, we have $\Xi_{9,10;4}(f)_1 = (6,6,7,7,8)$, $\Xi_{9,10;4}(f)_2 = \Phi_{5,10}(6,6,7,7,8) \oplus (1) = (1,1,2,2,3,1)$, $\Xi_{9,10;4}(f)_3 = \Phi_{6,10}(1,1,2,2,3,1) \oplus (1) = (5,5,6,6,7,5,1)$, $\Xi_{9,10;4}(f)_4 = \Phi_{7,10}(5,5,6,6,7,5,1) \oplus (2) = (1,1,2,2,3,1,10,2)$, and $\Xi_{9,10;4}(f) = \Xi_{9,10;4}(f)_5 = \Phi_{8,10}(1,1,2,2,3,1,10,2) \oplus (1) = (4,4,5,5,6,4,1,5,1)$.
\end{exm}

\begin{thm}
\label{cncorr}
$\Xi_{m,n;k} : B(m,n;k) \rightarrow PF(m,n)$ is a bijection for any $m \leq n \in \mathbb{N}$ and $0 \leq k \leq n-1$.
\end{thm}
\begin{proof}
It is easy to check that $\Xi_{m,n;k}(f) \in PF(m,n)$ for all $f \in B(m,n;k)$, where the occupied vertices in $\Xi_{m,n;k}(f)$ are
\begin{enumerate}[(a)]
\item the same as those of $f$ if $d$ is even
\item the reflection of those of $f$ if $d$ is odd
\end{enumerate}

We construct $\Xi_{m,n;k}^{-1}$ by reversing the procedure defining $\Xi_{m,n;k}$. Let $g \in PF(m,n)$.
We produce a sequence $\Xi_{m,n;k}^{-1}(g)_1$, $\Xi_{m,n;k}^{-1}(g)_2$, $\ldots$, $\Xi_{m,n;k}^{-1}(g)_a$ of progressively longer tuples with $\Xi_{m,n;k}^{-1}(g) = \Xi_{m,n;k}^{-1}(g)_a$. \\

Break up the $m$-tuple $g$ into sub-tuples $g = \tilde{g}_1 \oplus g_1$ satisfying one of the following:
\begin{enumerate}
\item\label{revcase1} All cars of $g_1$ have traverse path length not exceeding $k$, while the last car of $\tilde{g}_1$ has traverse path length at least $k+1$.
\item\label{revcase2} All cars of $g_1$ have traverse path length at least $k+1$, while the last car of $\tilde{g}_1$ has traverse path length not exceeding $k$.
\end{enumerate}
$\Xi_{m,n;k}^{-1}(g)_1$ will be a $\mathrm{len}(g_1)$-tuple defined in terms of $g_1$. In Case \ref{revcase1}, each car $c_1$ with preference $p_1$ in $g_1$ will have new preference $n+1-p_1$ in $\Xi_{m,n;k}^{-1}(g)_1$. In Case \ref{revcase2}, each car $c_1$ with preference $p_1$ in $g_1$ will have new preference $p_1+k$ in $\Xi_{m,n;k}^{-1}(g)_1$. \\

Partition $\Phi_{\mathrm{len}(g)-\mathrm{len}(g_1),n}(\tilde{g}_1)$ into sub-tuples $\Phi_{\mathrm{len}(g)-\mathrm{len}(g_1),n}(\tilde{g}_1) = \tilde{g}_2 \oplus g_2$ satisfying one of the following:
\begin{enumerate}[i.]
\item\label{revcase3} All cars of $g_2$ have traverse path length not exceeding $k$, while the last car of $\tilde{g}_2$ has traverse path length at least $k+1$.
\item\label{revcase4} All cars of $g_2$ have traverse path length at least $k+1$, while the last car of $\tilde{g}_2$ has traverse path length not exceeding $k$.
\end{enumerate}
Suppose Case \ref{revcase3} is satisfied. Each car $c_2$ with preference $p_2$ in $g_2$ will have new preference $n+1-p_2$. This yields a $\mathrm{len}(g_{2})$-tuple $\mu_2$ of preferences, and we define $\Xi_{m,n;k}^{-1}(g)_2 = \mu_2 \oplus \Xi_{m,n;k}^{-1}(g)_1$. Suppose Case \ref{revcase4} is satisfied. Each car $c_2$ with preference $p_2$ in $g_2$ will have new preference $p_2+k$. This yields a $\mathrm{len}(g_{2})$-tuple $\mu_2$ of preferences, and we define $\Xi_{m,n;k}^{-1}(g)_2 = \mu_2 \oplus \Xi_{m,n;k}^{-1}(g)_1$. \\

Partition $\Phi_{\mathrm{len}(g)-\mathrm{len}(\Xi_{m,n;k}^{-1}(g)_2),n}(\tilde{g}_2)$ into sub-tuples $\Phi_{\mathrm{len}(g)-\mathrm{len}(\Xi_{m,n;k}^{-1}(g)_2),n}(\tilde{g}_2) = \tilde{g}_3 \oplus g_3$ and continue this process until $\tilde{g}_a = \emptyset$ and we obtain $\Xi_{m,n;k}^{-1}(g)_a = \mu_a \oplus \Xi_{m,n;k}^{-1}(g)_{a-1}$.
\end{proof}
\begin{exm}
Consider the following directed path $I$. 

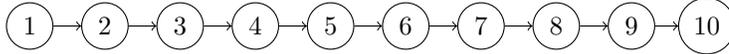
\begin{tikzpicture}
\node[shape=circle,draw=black] (1) at (0,0) {1};
\node[shape=circle,draw=black] (2) at (1,0) {2};
\node[shape=circle,draw=black] (3) at (2,0) {3};
\node[shape=circle,draw=black] (4) at (3,0) {4};
\node[shape=circle,draw=black] (5) at (4,0) {5};
\node[shape=circle,draw=black] (6) at (5,0) {6};
\node[shape=circle,draw=black] (7) at (6,0) {7};
\node[shape=circle,draw=black] (8) at (7,0) {8};
\node[shape=circle,draw=black] (9) at (8,0) {9};
\node[shape=circle,draw=black] (10) at (9,0) {10};

\path [->] (1) edge (2);
\path [->] (2) edge (3);
\path [->] (3) edge (4);
\path [->] (4) edge (5);
\path [->] (5) edge (6);
\path [->] (6) edge (7);
\path [->] (7) edge (8);
\path [->] (8) edge (9);
\path [->] (9) edge (10);
\end{tikzpicture}

For $k=4$ and $g = (2,2,3,4,3,2,3,3) = (2,2,3,4,3) \oplus (2,3,3) \in P(8,10)$, we have $\Xi_{8,10;4}^{-1}(g)_1$ $= (6,7,7)$ and $\Xi_{8,10;4}^{-1}(g) = \Xi_{8,10;4}^{-1}(g)_2 = (6,6,5,4,5) \oplus (6,7,7) = (6,6,5,4,5,6,7,7)$.

For $k=3$ and $g = (1,2,4,3,5,1,5) = (1,2,4,3,5) \oplus (1) \oplus (5) \in PF(7,10)$, we have $\Xi_{7,10;3}^{-1}(g)_1 = (6)$, $\Xi_{7,10;3}^{-1}(g)_2 = (8) \oplus (6) = (8,6)$, and $\Xi_{7,10;3}^{-1}(g) = \Xi_{7,10;3}^{-1}(g)_3 = (5,6,8,7,9) \oplus (8,6) = (5,6,8,7,9,8,6)$.
\end{exm}

Now we introduce some notations that will be used in Section \ref{obstr bound}, which can be read immediately after without referring to Section \ref{conta ties}. Fix $f \in PF(m,n;k)$. For each car $c \in [m]$, let $\mathrm{tpath}^N(c)$ denote the (undirected) path car $c$ traverses before parking; we call $\mathrm{tpath}^N(c)$ the \textbf{Naples traverse path of $c$}. In this case, $\mathrm{tpath}^N(c)$ will also include vertices left of the preferred spot if car $c$ has to travel backward.

For vertices $i, j \in [n]$, we write $i \sim^N j$ if there exist cars $c_1, c_2, \ldots, c_a$ such that $\mathrm{tpath}^N(c_1)$ contains $i$, $\mathrm{tpath}^N(c_a)$ contains $j$, and $\mathrm{tpath}^N(c_l) \cap \mathrm{tpath}^N(c_{l+1}) \neq \emptyset$. The equivalence relation $\sim^N$ yields ``path components'' that we shall call \textbf{Naples components}. We denote a Naples component as a path, which is the union of all the relevant Naples traverse paths.

\begin{exm}
Let $k=4$ and consider $f = (4,4,7,1,1,9,10,10,1) \in PF(9,10;4)$.

We have $\mathrm{tpath}^N(5) = (1,2)$, $\mathrm{tpath}^N(8) = (8,9,10)$, and $\mathrm{tpath}^N(9) = (1,2,3,4,5)$.

The Naples components of $f$ are $(1,2,3,4,5)$, $(7)$, and $(8,9,10)$. 
\end{exm}

\subsection{$\Xi_{m,n;k}$ Preserves Ties in $B(m,n;k)$}
\label{conta ties}
Even though the correspondence $\Xi_{m,n;k}$ does not necessarily preserve the number of ties in a given $f \in B(m,n;k)$, we can show that $B(m,n;k)$ and $\Xi_{m,n;k}[B(m,n;k)] = PF(m,n)$ have the same number of ties, from which it follows that $B_{n,k}$ and $PF_n$ do as well. For $f = f_1 \oplus f_2 \oplus \cdots \oplus f_d \in B(m,n;k)$ in $k$-decomposition, observe the following about how the number of ties may differ between $f$ and $\Xi_{m,n;k}(f)$. An increase or decrease in the number of ties as a result of applying $\Xi_{m,n;k}$ can only occur at the car pairs $(a-1,a)$, for some $a \in [m]$ such that $f(a)$ is the left endpoint of some sub-tuple $f_i$. Hence we only need to examine the cars $b_1,b_2,\ldots,b_{d-1}$, where $f(b_i)$ is the left endpoint of $f_{i+1}$, for possible changes to the number of ties. 

To record the change in ties due to $\Xi_{m,n;k}$, we define the \textbf{tie-change tuple} as $\Delta\mathrm{ties}(\Xi_{m,n;k},f)$ $= (e_1,e_2,\ldots,e_{d-1})$, where 
\[   
e_i = 
     \begin{cases}
       1 &\quad\text{if }(b_i-1,b_i) \text{ is a tie in }\Xi_{m,n;k}(f) \text{ but not in }f \\
      -1 &\quad\text{if }(b_i-1,b_i) \text{ is a tie in }f \text{ but not in }\Xi_{m,n;k}(f)\\
       0 &\quad\text{if }(b_i-1,b_i) \text{ is a tie in both } f \text{ and }\Xi_{m,n;k}(f) \text{ or a tie in neither of them}\\
     \end{cases}.
\]
For even $i$, $(b_i-1,b_i)$ cannot be a tie in $f$ because car $b_i-1$ parks forward whereas car $b_i$ parks backward, which means that $(b_i-1,b_i)$ cannot be a tie in $\Xi_{m,n;k}(f)$ either. Hence $e_i = 0$ whenever $i$ is even.

Let $a < b \in [n]$ where $a$ and $b$ are the left and right endpoints of Naples components respectively. We introduce an involution $\Phi_{m,n}^{a,b}$ on $B(m,n;k)$ that reflects the Naples components on $\mathrm{path}(a,b)$, in the same spirit as $\Phi_{m,n}$. Let $L = (i,i+1,i+2,\ldots,i+j)$ be any Naples component on $\mathrm{path}(a,b)$. For any car $c$ preferring $i+d'$ in $f$, for $0 \leq d' \leq j$, $c$ will prefer $b+a-i-j+d'$ in $\Phi_{m,n}^{a,b}(f)$. Any car $c'$ with preference outside $\mathrm{path}(a,b)$ in $f$ will have the same preference in $\Phi_{m,n}^{a,b}(f)$. 

We will show that $B(m,n;k)$ and $PF(m,n)$ have the same number of ties, by defining an involution $\Psi_{m,n;k}$ on $B(m,n;k)$ with the property that $$\Delta\mathrm{ties}(\Xi_{m,n;k},\Psi_{m,n;k}(f)) = -\Delta\mathrm{ties}(\Xi_{m,n;k},f).$$ To aid us in this goal, we first define a simple involution $\psi_{m,n;k}$ on $B(m,n;k)$ that only negates the last component $e_{d-1}$ of $\Delta\mathrm{ties}(\Xi_{m,n;k},f)$.

For simplicity of notation, we denote by $f \mid' i$ the restriction $f \mid [\sum_{j=1}^{i}{\mathrm{len(f_j)}}]$; in other words, $f \mid' i$ is $f$ restricted to the cars before $f_{i+1}$. For any car $c \in [m]$, let $\mathrm{tcomp}(c)$ denote the concatenation of all the Naples components that intersect $\mathrm{tpath}^N(c)$ but are not strictly right of $f(c)$, and let $c^<$ and $c^>$ denote the left and right endpoints of $\mathrm{tcomp}(c)$ respectively.

\begin{exm}
Let $m=10$, $n=10$, and $k=3$.

For $f = (5,5,5,5,4,4,8,8,8,8) \in B(10,10;3)$, we have $f = (5,5,5,5,4) \oplus (4) \oplus (8,8) \oplus (8,8)$ in $3$-decomposition, $\mathrm{tpath}^N(6) = (1,2,3,4,5,6)$, $f \mid' 3 = (5,5,5,5,4,4,8,8)$, and the Naples components of $f \mid' 3$ are $(1,2,3,4,5,6)$ and $(7,8)$. We have $b_3 = 9$, $\Xi_{10,10;3}(f) = (1,1,1,1,2,1,7,7,5,5)$, and $\Delta\mathrm{ties}(\Xi_{10,10;3},f) = (-1,0,-1)$. Car $9$ has $\mathrm{tcomp}(9) = (1,2,3,4,5,$ $6) \oplus (7,8) = (1,2,3,4,5,6,7,$ $8)$, $9^< = 1$, and $9^> = 8$. We have $\Phi_{10,10}^{9^<,9^>}(f\mid'3) = (7,7,7,7,6,6,2,$ $2)$, whose Naples components are $(1,2)$ and $(3,4,5,6,7,8)$.
\end{exm}

\begin{exm}
Let $m=10$, $n=11$, and $k=3$.

For $f = (5,5,5,5,4,4,8,9,8,8) \in B(10,11;3)$, we have $f = (5,5,5,5,4) \oplus (4) \oplus (8,9,8) \oplus (8)$ in $3$-decomposition, $\mathrm{tpath}^N(10) = (5,6,7,8,9,10)$. We have $b_3 = 10$, $\Xi_{10,11;3}(f) = (1,1,1,1,2,1,7,$ $9,7,5)$, and $\Delta\mathrm{ties}(\Xi_{10,11;3},f) = (-1,0,-1)$. Car $10$ has $\mathrm{tcomp}(10) = (1,2,3,4,5,6) \oplus (7,8) = (1,2,3,4,5,6,7,8)$, $10^< = 1$, and $10^> = 8$. Note that the Naples component $(9)$ of $f \mid' 3$ is not included in $\mathrm{tcomp}(10)$ because it is strictly to the right of $f(10) = 8$.
\end{exm}

To give $\psi_{m,n;k}(f)$, we consider the Naples traverse path of car $b_{d-1}$; for a rough picture of how $\psi_{m,n;k}$ acts on $f$, see Remark \ref{lastneg}. We define
\[   
\psi_{m,n;k}(f) = 
     \begin{cases}
       \Phi_{m,n}^{b_{d-1}^<,b_{d-1}^>}(f \mid' d-1) \oplus \gamma(b_{d-1})_1 &\quad\text{if } e_{d-1} = 1\\
       \Phi_{m,n}^{b_{d-1}^<,b_{d-1}^>}(f \mid' d-1) \oplus \gamma(b_{d-1})_2 &\quad\text{if } e_{d-1} = -1\\
       f &\quad\text{if } e_{d-1} = 0\\
     \end{cases},
\]
where the $\mathrm{len}(f_d)$-tuples $\gamma(b_{d-1})_1$, $\gamma(b_{d-1})_2$ are defined by their entries (function values) as follows. For any car $c \geq b_{d-1}$, we have 
\[   
\gamma(b_{d-1})_1(c) = 
     \begin{cases}
        \Phi_{m,n}^{b_{d-1}^<,b_{d-1}^>}(f \mid' d-1)(b_{d-1}-1) &\quad\text{if } f(c) = f(b_{d-1}) \\
       f_d(c) &\quad\text{otherwise} \\
     \end{cases}\
\] 
and
\[   
\gamma(b_{d-1})_2(c) = 
     \begin{cases}
       \mathrm{aim}(b_{d-1}-1,f) &\quad\text{if } f(c) = f(b_{d-1}) \\
       f_d(c) &\quad\text{otherwise} \\
     \end{cases},\
\] 
where $\mathrm{aim}(b_{d-1}-1,f) = \Phi_{m,n}^{b_{d-1}^<,b_{d-1}^>}(f \mid' d-1)(b_{d-1}-1)+[\Xi_{m,n;k}(f)(b_{d-1}-1)-\Xi_{m,n;k}(f)(b_{d-1})]$ is the unique vertex that car $b_{d-1}$ can prefer in $\psi_{m,n;k}(f)$ so that $(b_{d-1}-1,b_{d-1})$ is a tie in $\Xi_{m,n;k}[\psi_{m,n;k}(f)]$, as will be shown in Lemma \ref{aim exists}.

In the degenerate case $d=1$, we define $\psi_{m,n;k}(f) = f$, as $f$ and $\Xi_{m,n;k}(f)$ have the same ties.

\begin{rem}
\label{lastneg}
$\psi_{m,n;k}$ fixes $f$ if $e_{d-1} = 0$, as $\Xi_{m,n;k}$ caused no tie changes in the last component. Roughly speaking, for $e_{d-1} \neq 0$, $\psi_{m,n;k}$ reflects (via $\Phi_{m,n}^{b_{d-1}^<,b_{d-1}^>}$) the Naples components (except the ones after $f(b_{d-1})$) of $f \mid' d-1$ traversed by car $b_{d-1}$, then it decides the preferences of the remaining cars, i.e. those of $f_d$, depending on the sign of $e_{d-1}$:
\begin{enumerate}
\item If $e_{d-1} = 1$, that means $\Xi_{m,n;k}$ added a tie in the last component, so the preferences of the $f_d$ cars are decided by $\gamma(b_{d-1})_1$, which ``aims'' certain cars in such a way that results in $\Xi_{m,n;k}$ subtracting a tie and which fixes the other car preferences.
\item If $e_{d-1} = -1$, that means $\Xi_{m,n;k}$ subtracted a tie in the last component, so the preferences of the $f_d$ cars are decided by $\gamma(b_{d-1})_2$, which ``aims'' certain cars in such a way that results in $\Xi_{m,n;k}$ adding a tie and which fixes the other car preferences.
\end{enumerate}
\end{rem}

\begin{lem}
\label{aim exists}

$\psi_{m,n;k}$ is an involution on $B(m,n;k)$ with the following properties. For $f \in B(m,n;k)$ with $k$-decomposition $f = f_1 \oplus f_2 \oplus \cdots \oplus f_d$ where $d \geq 2$ and $\Delta\mathrm{ties}(\Xi_{m,n;k},f) = (e_1,e_2,\ldots,e_{d-2},e_{d-1})$, we have $\Delta\mathrm{ties}(\Xi_{m,n;k},\psi_{m,n;k}(f)) = (e_1,e_2,\ldots,e_{d-2},-e_{d-1})$. 
Furthermore, $\psi_{m,n;k}(f)$ has the same Naples components as $f$, and $\psi_{m,n;k}(f)$ has $k$-decomposition $\psi_{m,n;k}(f) = f'_1 \oplus f'_2 \oplus \ldots \oplus f'_d$ where $\mathrm{len}(f'_i) = \mathrm{len}(f_i)$.
\end{lem}
\begin{proof}
Clearly, $f \mid' d-1$ and $\Phi_{m,n}^{b_{d-1}^<,b_{d-1}^>}(f \mid' d-1)$ have exactly the same ties. It remains to compare the last entries of $\Delta\mathrm{ties}(\Xi_{m,n;k},f)$ and $\Delta\mathrm{ties}(\Xi_{m,n;k},\psi_{m,n;k}(f))$, to which end it suffices to check the cases $e_{d-1} = \pm1$, for which $d-1$ must necessarily be odd and hence the cars of $f_d$ must park forward. It is easy to see that the parking components of $\Xi_{m,n;k}(f) \mid' d-1$ coincide with the Naples components of $f \mid' d-1$.

Consider the Naples component $L = (i,i+1,i+2,\ldots,i+j)$ in $f \mid' d-1$ containing $f(b_{d-1}-1) = i+a$ for some $0 \leq a \leq j$, with car $b_{d-1}-1$ eventually parking (backward) at vertex $i$. $L$ is also a parking component in $\Xi_{m,n;k}(f) \mid' d-1$ containing $\Xi_{m,n;k}(f)(b_{d-1}-1) = i+j-a$, with car $b_{d-1}-1$ eventually parking at vertex $i+j$. Since $e_{d-1} = \pm1$, there is a tie change occurring between cars $b_{d-1}$ and $b_{d-1}-1$, so $L$ must be contained in $\mathrm{tcomp}(b_{d-1})$ and hence reflected by $\Phi_{m,n}^{b_{d-1}^<,b_{d-1}^>}$.

Let $L' = (i',i'+1,i'+2,\ldots,i'+j)$ be the unique Naples component corresponding to $L$ in $\Phi_{m,n}^{b_{d-1}^<,b_{d-1}^>}(f \mid' d-1) = \psi_{m,n;k}(f)\mid'd-1$, where $L'$ is preferred by the same set of cars, with $\psi_{m,n;k}(f)(b_{d-1}-1) = i'+a$. $L'$ is also a parking component in $\Xi_{m,n;k}[\psi_{m,n;k}(f)] \mid' d-1$, with car $b_{d-1}-1$ preferring $i'+j-a$. 

Suppose $e_{d-1} = -1$, in which case $f(b_{d-1}) = i+a$. Since $(b_{d-1}-1,b_{d-1})$ is not a tie in $\Xi_{m,n;k}(f)$ and since the traverse path length of car $b_{d-1}-1$ is at most $k$, we have $\Xi_{m,n;k}(f)(b_{d-1}) = i+a-k \leq i$ and $\Xi_{m,n;k}(f)(b_{d-1}) \leq i+j-a$. We have $\Xi_{m,n;k}(f)(b_{d-1}-1)-\Xi_{m,n;k}(f)(b_{d-1}) = j-2a+k > 0$. Now, for any car $c \geq b_{d-1}$ preferring vertex $f(b_{d-1})$ in $f$, $c$ will prefer $\mathrm{aim}(b_{d-1}-1,f) = \Phi_{m,n}^{b_{d-1}^<,b_{d-1}^>}(f \mid' d-1)(b_{d-1}-1)+[\Xi_{m,n;k}(f)(b_{d-1}-1)-\Xi_{m,n;k}(f)(b_{d-1})] = i'+a+(j-2a+k) = i'+j-a+k$ in $\psi_{m,n;k}(f)$. Notice that $i-(i+a-k) = k-a = i'+j-a+k-(i'+j)$, meaning that the preference $\psi_{m,n;k}(f)(c)$ makes the car $c$ travel the same number $k-a$ of steps left to enter $L'$ in $\psi_{m,n;k}(f)$ as it does right to enter $L$ in $\Xi_{m,n;k}(f)$, and hence $c$ does park forward in $\psi_{m,n;k}(f)$ because the same number $k-a$ of occupied vertices precede $i$ and follow $i'+j$. It also follows that $\mathrm{tpath}^N(c)$ intersects the same Naples components in $f$ and in $\psi_{m,n;k}(f)$, up to permutation, so $\psi_{m,n;k}(f)$ has the same Naples components as $f$. Since $\Xi_{m,n;k}[\psi_{m,n;k}(f)](b_{d-1}-1) = i'+j-a$ and $\Xi_{m,n;k}[\psi_{m,n;k}(f)](b_{d-1}) = i'+j-a+k-k = i'+j-a$, $(b_{d-1}-1,b_{d-1})$ is indeed a tie in $\Xi_{m,n;k}[\psi_{m,n;k}(f)]$ as desired. 

Suppose $e_{d-1} = 1$. For $\Xi_{m,n;k}(f)(b_{d-1}) = \Xi_{m,n;k}(f)(b_{d-1}-1) = i+j-a$ to happen, car $b_{d-1}$ must prefer $i+j-a+k \geq i+j$ in $f$. We have $\psi_{m,n;k}(f)(b_{d-1}-1) = i'+a = \psi_{m,n;k}(f)(c)$, $\Xi_{m,n;k}[\psi_{m,n;k}(f)](b_{d-1}-1) = i'+j-a$, and $\Xi_{m,n;k}[\psi_{m,n;k}(f)](c) = i'+a-k \leq i'$, for any car $c \geq b_{d-1}$ preferring vertex $f(b_{d-1})$ in $f$. It follows that $i'-(i'+a-k) = k-a = i+j-a+k-(i+j)$, so the preference $\psi_{m,n;k}(f)(c)$ makes the car $c$ travel the same number $k-a$ of steps left to enter $L$ in $f$ as it does right to enter $L'$ in $\Xi_{m,n;k}[\psi_{m,n;k}(f)]$, and hence $\psi_{m,n;k}(f)$ has the same Naples components as $f$. Finally, $\Xi_{m,n;k}[\psi_{m,n;k}(f)](b_{d-1}-1)-\Xi_{m,n;k}[\psi_{m,n;k}(f)](b_{d-1}) = i'+j-a-(i'+a-k) = j-2a+k = f(b_{d-1})-f(b_{d-1}-1) > 0$ since $(b_{d-1}-1,b_{d-1})$ is not a tie in $f$. Thus, $(b_{d-1}-1,b_{d-1})$ is indeed not a tie in $\Xi_{m,n;k}[\psi_{m,n;k}(f)]$ as desired. 

In both cases $e_{d-1} = \pm1$, the endpoints $b_{d-1}^<$ and $b_{d-1}^>$ are the same and the Naples components enclosed are in reverse orders, so the choices for $\psi_{m,n;k}(f)(b_{d-1})$ are uniquely determined. Hence $\psi_{m,n;k}$ is indeed an involution.
\end{proof}

Furthermore, define $\mathrm{out}(\psi_{m,n;k},f)$ to be the $\mathrm{len}(f_d)$-tuple 
\[   
\mathrm{out}(\psi_{m,n;k},f) = 
     \begin{cases}
       \gamma(b_{d-1})_1 &\quad\text{if } e_{d-1} = 1\\
       \gamma(b_{d-1})_2 &\quad\text{if } e_{d-1} = -1\\
       f_d &\quad\text{if } e_{d-1} = 0\\
     \end{cases},
\]
which we will call the \textbf{output tail of $\psi_{m,n;k}$}; this is just the last sub-tuple of $\psi_{m,n;k}(f)$.

\begin{exm}
Let $m=10$, $n=11$, and $k=3$.

For $f = (5,5,5,5,4,4,8,9,8,8) \in B(10,11;3)$, we have $\Xi_{10,11;3}(f) = (1,1,1,1,2,1,7,9,7,5)$, $\psi_{10,11;3}(f) = (7,7,7,7,6,6,2,9,2,4)$, and $\Xi_{10,11;3}[\psi_{10,11;3}(f)] = (3,3,3,3,4,3,1,9,1,1)$. Notice that $\Delta\mathrm{ties}(\Xi_{10,11;3},f) = (-1,0,-1)$, $\Delta\mathrm{ties}(\Xi_{10,11;3},\psi_{10,11;3}(f)) = (-1,0,1)$, and $\mathrm{out}(\psi_{10,11;3},f)$ $= (4)$.
\end{exm}

Now we introduce the involution $\Psi_{m,n;k}$ on $B(m,n;k)$; see Remark \ref{tieswitchrem} for a rough picture of how $\Psi_{m,n;k}$ works. First, we define the following to simplify notation:
\begin{enumerate}
\item $\psi^0(f) = f$
\item $\psi^1(f) = \psi_{m,n;k}(f)\mid'd-1$
\item $\psi^2(f) = \psi_{m-\mathrm{len}(f_d),n;k}(\psi^1(f))\mid'd-2$ 
\item $\psi^3(f) = \psi_{m-\mathrm{len}(f_d)-\mathrm{len}(f_{d-1}),n;k}(\psi^2(f))\mid'd-3$
\item In general, given $\psi^l(f)$, define \[\psi^{l+1}(f) = \psi_{m-\sum_{i=0}^{l-1}{\mathrm{len}(f_{d-i})},n;k}(\psi^l(f))\mid'd-l-1.\]
\end{enumerate}

Finally, we define $\Psi_{m,n;k}$. Given $f \in B(m,n;k)$ with $k$-decomposition $f = f_1 \oplus f_2 \oplus \cdots \oplus f_d$, we construct $\Psi_{m,n;k}(f)$ in stages, by applying in succession $\psi_{i,n;k}$ for various $i \in [m]$:
\begin{enumerate}
\item $\Psi_{m,n;k}(f)_1 = \mathrm{out}(\psi_{m,n;k},f)$. 
\item $\Psi_{m,n;k}(f)_2 = \mathrm{out}(\psi_{m-\mathrm{len}(f_d),n;k},\psi^1(f)) \oplus \Psi_{m,n;k}(f)_1$
\item $\Psi_{m,n;k}(f)_3 = \mathrm{out}(\psi_{m-\mathrm{len}(f_d)-\mathrm{len}(f_{d-1}),n;k},\psi^2(f)) \oplus \Psi_{m,n;k}(f)_2$
\item In general, given $\Psi_{m,n;k}(f)_j$, we have $$\Psi_{m,n;k}(f)_{j+1} = \mathrm{out}(\psi_{m-\sum_{i=0}^{j-1}{\mathrm{len}(f_{d-i})},n;k},\psi^j(f)) \oplus \Psi_{m,n;k}(f)_j.$$
\item Finally, we have $\Psi_{m,n;k}(f) = \Psi_{m,n;k}(f)_{d}$.
\end{enumerate}

\begin{rem}
\label{tieswitchrem}
Roughly speaking, given $f \in B(m,n;k)$, $\Psi_{m,n;k}(f)$ is defined recursively through repeated applications of $\psi_{i,n;k}$ for various $i \in [m]$, in the following manner (where we have omitted the cumbersome indices):
\begin{enumerate}
\item\label{tieswitch1} Apply $\psi$ to $f$, then remove the output tail.
\item\label{tieswitch2} Apply $\psi$ to the result of Step \ref{tieswitch1}., then remove the output tail.
\item Apply $\psi$ to the result of Step \ref{tieswitch2}., then remove the output tail.
\item Continue this process until the empty tuple is reached, then concatenate all these removed output tails in right-to-left order to get $\Psi_{m,n;k}(f)$.
\end{enumerate}
\end{rem}

\begin{thm}
\label{ties invo}
$\Psi_{m,n;k}$ is an involution on $B(m,n;k)$ with the following property. For all $f \in B(m,n;k)$ we have $\Delta\mathrm{ties}(\Xi_{m,n;k},\Psi_{m,n;k}(f)) = -\Delta\mathrm{ties}(\Xi_{m,n;k},f)$.
\end{thm}
\begin{proof}
Let $f \in B(m,n;k)$ with $k$-decomposition $f = f_1 \oplus f_2 \oplus \cdots \oplus f_d$. For each $j \in [d]$, applying $\psi_{m-\sum_{i=0}^{j-1}{\mathrm{len}(f_{d-i})},n;k}$ to $\psi^j(f)$ negates the last entry of $\Delta\mathrm{ties}(\Xi_{m-\sum_{i=0}^{j-1}{\mathrm{len}(f_{d-i})},n;k},\psi^j(f))$ while fixing other entries, and yields the same occupied vertices, Naples components, and $k$-decomposition as $\psi^j(f)$, by Lemma \ref{aim exists}. 
Hence $\Psi_{m,n;k}$ negates all the entries of $\Delta\mathrm{ties}(\Xi_{m,n;k},f)$ as desired, upon outputting $\Psi_{m,n;k}(f) = \Psi_{m,n;k}(f)_{d}$ via the above procedure.

To show that $\Psi_{m,n;k}^2(f) = f$, it suffices to run $\Psi_{m,n;k}(f)$ through the aforementioned procedure, applying the various involutions $\psi_{i,n;k}$ in the order prescribed. By Lemma \ref{aim exists}, we will obtain $\Psi_{m,n;k}(\Psi_{m,n;k}(f))_j = f_{d-j+1} \oplus \ldots \oplus f_{d-1} \oplus f_{d}$. Thus, we have $\Psi_{m,n;k}(\Psi_{m,n;k}(f))_d = f$ as desired.
\end{proof}
\begin{exm}
Let $m=8$, $n=10$, and $k=3$.

For $f = (6,6,6,6,5,5,4,4) \in B(8,10;3)$, we have $\Xi_{8,10;3}(f) = (1,1,1,1,2,1,4,1)$. Since $\Delta\mathrm{ties}(\Xi_{8,10;3},f) = (-1,0,-1)$, $\Xi_{8,10;3}(f)$ has two less ties than $f$. Then $\Psi_{8,10;3}(f)_1 = (7)$, $\Psi_{8,10;3}(f)_2 = (4,7)$, $\Psi_{8,10;3}(f)_3 = (6,4,7)$, and 
$\Psi_{8,10;3}(f) = \Psi_{8,10;3}(f)_4 = (6,6,6,6,5,6,4,7) \in B(8,10;3)$, and we have $\Xi_{8,10;3}(\Psi_{8,10;3}(f)) = (1,1,1,1,2,2,4,4)$. Since $\Delta\mathrm{ties}(\Xi_{8,10;3},\Psi_{8,10;3}(f))$ $= (1,0,1)$, $\Xi_{8,10;3}(\Psi_{8,10;3}(f))$ has two more ties than $\Psi_{8,10;3}(f)$.

For $f = (6,6,6,6,5,6,8,8) \in B(8,10;3)$, we have $\Xi_{8,10;3}(f) = (2,2,2,2,3,3,8,5)$. Since $\Delta\mathrm{ties}(\Xi_{8,10;3},f) = (1,0,-1)$, $\Xi_{8,10;3}(f)$ has the same number of ties as $f$. Then $\Psi_{8,10;3}(f)_1 = (5)$, $\Psi_{8,10;3}(f)_2 = (2,5)$, $\Psi_{8,10;3}(f)_3 = (6,2,5)$, and
$\Psi_{8,10;3}(f) = \Psi_{8,10;3}(f)_4 = (7,7,7,7,6,6,2,5)$ $\in B(8,10;3)$, we have $\Xi_{8,10;3}(\Psi_{8,10;3}(f)) = (3,3,3,3,4,3,2,2)$. Since $\Delta\mathrm{ties}(\Xi_{8,10;3},\Psi_{8,10;3}(f)) = (-1,0,1)$, $\Xi_{8,10;3}(\Psi_{8,10;3}(f))$ has the same number of ties as $\Psi_{8,10;3}(f)$.
\end{exm}
\begin{cor}
\label{ties equal}
$B(m,n;k)$ has the same number of ties as $PF(m,n)$ for any $0 \leq k \leq n-1$.
\end{cor}

\section{Injection into the parking functions with obstructions}
\label{obstr bound}
Let $PF(m,n;k)$ denote the set of $k$-Naples parking functions with $m$ cars on $n$ vertices. We now extend $\Xi_{m,n;k}$ to an injection of $PF(m,n;k)$ into $LPF(m,n+k;k)$, the set of parking functions with $m$ cars on $n+k$ vertices where the first $k$ parking spots are already occupied by obstructions. Again, let $I$ denote the directed path with $n+k$ vertices. Let $OPF(m,n+k;k) \supset LPF(m,n+k;k)$ denote the set of parking functions with $m$ cars on $n+k$ vertices where certain $k$ consecutive vertices are obstructed. We will first need to generalize the involution $\Phi_{i,n}$ for $i \leq m$ to one on $OPF(m,n+k;k)$.

Let $f \in OPF(m,n+k;k)$. We can define traverse paths of cars just as in $PF(m,n)$, treating the obstructions as parked cars. Then we can define parking components in the same manner. Let $\bar{k}$ denote the path consisting of the $k$ obstructed vertices. An \textbf{obstruction component} of $f$ is a path on $I$ that is either of the following:
\begin{enumerate}
\item a parking component of $f$ that does not intersect $\bar{k}$
\item the union of $\bar{k}$ with all parking components of $f$ that intersect $\bar{k}$.
\end{enumerate}

We now generalize $\Phi_{m,n}$ to an involution $\bar{\Phi}_{m,n+k;k}$ on $OPF(m,n+k;k)$ as follows. Given $f \in OPF(m,n+k;k)$, define $\bar{\Phi}_{m,n+k;k}(f)$ as follows. Let $L = (i,i+1,i+2,\ldots,i+j)$ be an obstruction component of $f$, with its reflection $(n-i+1,n-i,n-i-1,\ldots,n-i-j+1)$ which we rewrite as the path $(n-i-j+1,n-i-j+2,\ldots,n-i+1)$. For any car $c$ preferring $i+a$ in $f$ for $0 \leq a \leq j$, $c$ will prefer $n-i-j+a+1$ in $\bar{\Phi}_{m,n+k;k}(f)$. In addition, if $i+a$, $i+a+1$, $\ldots$, $i+a+k-1$ are the obstructed vertices in $f$ for $0 \leq a \leq j$, then the obstructed vertices in $\bar{\Phi}_{m,n+k;k}(f)$ will be $n-i-j+a+1$, $n-i-j+a+2$, $\ldots$, $n-i-j+a+k$. See Remark \ref{geninv} for a rough picture of $\bar{\Phi}_{m,n+k;k}$. Thus defined, it is easy to see that $\bar{\Phi}_{m,n+k;k}$ is an involution preserving the traverse path length of each car.

\begin{rem}
\label{geninv}
For any obstruction component $L$, $\bar{\Phi}_{m,n+k;k}$ maps $L$ to its reflection $L' = n+k+1-L$, since $I$ has length $n+k$. Furthermore, if $L' = L+b$, then any preference $b_i \in L$ is shifted to $b_i+b$ by $\bar{\Phi}_{m,n+k;k}$.
\end{rem}

In the following examples, we will indicate an obstructed vertex by marking its label with a prime (').

\begin{exm}
Consider the following directed path $I$. \\ 
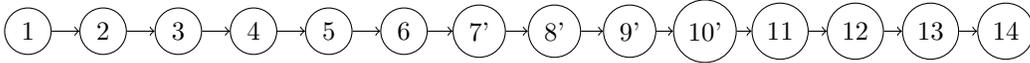
\begin{tikzpicture}
\node[shape=circle,draw=black] (1) at (0,0) {1};
\node[shape=circle,draw=black] (2) at (1,0) {2};
\node[shape=circle,draw=black] (3) at (2,0) {3};
\node[shape=circle,draw=black] (4) at (3,0) {4};
\node[shape=circle,draw=black] (5) at (4,0) {5};
\node[shape=circle,draw=black] (6) at (5,0) {6};
\node[shape=circle,draw=black] (7) at (6,0) {7'};
\node[shape=circle,draw=black] (8) at (7,0) {8'};
\node[shape=circle,draw=black] (9) at (8,0) {9'};
\node[shape=circle,draw=black] (10) at (9,0) {10'};
\node[shape=circle,draw=black] (11) at (10,0) {11};
\node[shape=circle,draw=black] (12) at (11,0) {12};
\node[shape=circle,draw=black] (13) at (12,0) {13};
\node[shape=circle,draw=black] (14) at (13,0) {14};

\path [->] (1) edge (2);
\path [->] (2) edge (3);
\path [->] (3) edge (4);
\path [->] (4) edge (5);
\path [->] (5) edge (6);
\path [->] (6) edge (7);
\path [->] (7) edge (8);
\path [->] (8) edge (9);
\path [->] (9) edge (10);
\path [->] (10) edge (11);
\path [->] (11) edge (12);
\path [->] (12) edge (13);
\path [->] (13) edge (14);
\end{tikzpicture}

For $f = (1,1,11,10,10,14,6,2,4,4) \in OPF(10,14;4)$, the obstruction components are $(1,2,3), (4,5), (6)$, $(7,8,9,10,11,12,13), (14)$, and we have $\bar{\Phi}_{10,14;4}(f) = (12,12,6,5,5,1,9,13,$ $10,10)$ on \\ 
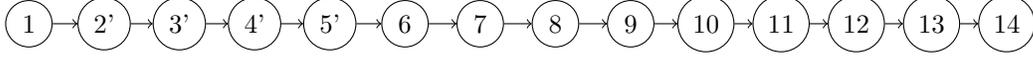
\begin{tikzpicture}
\node[shape=circle,draw=black] (1) at (0,0) {1};
\node[shape=circle,draw=black] (2) at (1,0) {2'};
\node[shape=circle,draw=black] (3) at (2,0) {3'};
\node[shape=circle,draw=black] (4) at (3,0) {4'};
\node[shape=circle,draw=black] (5) at (4,0) {5'};
\node[shape=circle,draw=black] (6) at (5,0) {6};
\node[shape=circle,draw=black] (7) at (6,0) {7};
\node[shape=circle,draw=black] (8) at (7,0) {8};
\node[shape=circle,draw=black] (9) at (8,0) {9};
\node[shape=circle,draw=black] (10) at (9,0) {10};
\node[shape=circle,draw=black] (11) at (10,0) {11};
\node[shape=circle,draw=black] (12) at (11,0) {12};
\node[shape=circle,draw=black] (13) at (12,0) {13};
\node[shape=circle,draw=black] (14) at (13,0) {14};

\path [->] (1) edge (2);
\path [->] (2) edge (3);
\path [->] (3) edge (4);
\path [->] (4) edge (5);
\path [->] (5) edge (6);
\path [->] (6) edge (7);
\path [->] (7) edge (8);
\path [->] (8) edge (9);
\path [->] (9) edge (10);
\path [->] (10) edge (11);
\path [->] (11) edge (12);
\path [->] (12) edge (13);
\path [->] (13) edge (14);
\end{tikzpicture}
\end{exm}

Also, there is a simple injection $\iota_{k}: PF(m,n) \rightarrow LPF(m,n+k;k)$ given as follows. Let $f \in PF(m,n)$. For each car $c_i$ with preference $p_i$ in $f$, $c_i$ will have preference $p_i+k$ in $\iota_{k}(f)$. Roughly speaking, $\iota_k$ just adds $k$ obstructed vertices to the left.

Now we use $\bar{\Phi}_{i,n+k;k}$ for $i \in [m]$ and $\iota_{k}$ to construct an injection $\bar{\Xi}_{m,n;k}: PF(m,n;k) \rightarrow LPF(m,n+k;k)$ which extends $\Xi_{m,n;k}(f)$. For any $f \in PF(m,n;k)-B(m,n;k)$, the cars of $f$ that must ``exit'' the parking lot in order to complete their backward search will correspond to the cars of $\bar{\Xi}_{m,n;k}(f) \in LPF(m,n+k;k)$ that prefer obstructed vertices. Think of $\bar{\Xi}_{m,n;k}$ as extending the parking lot so that all cars can finish their backward search without exiting.

Let $f \in PF(m,n;k)$ with its $k$-decomposition $f = f_1 \oplus f_2 \oplus \cdots \oplus f_d$. If $f \in B(m,n;k)$, then define $$\bar{\Xi}_{m,n;k}(f) = \iota_k[\Xi_{m,n;k}(f)].$$ We now define $\bar{\Xi}_{m,n;k}(f)$ for the case $f \in PF(m,n;k)-B(m,n;k)$; see Remark \ref{obstrem} for a rough picture of the procedure.

As in Section \ref{conta bij}, we construct $\bar{\Xi}_{m,n;k}(f)$ in stages by specifying the new preferences of the cars for each of $f_1, f_2, \ldots$ in that order; denote by $\bar{\Xi}_{m,n;k}(f)_i$ the parking function obtained after $f_i$. 
\begin{enumerate}
\item For each car $c_1$ with preference $p_1$ in $f_1$, $c_1$ will have new preference $n+1-p_1$. This yields $\bar{\Xi}_{m,n;k}(f)_1 \in PF(\mathrm{len}(f_1),n)$.
\item We define $\bar{\Xi}_{m,n;k}(f)_2 = \iota_k[\Phi_{\mathrm{len}(f_1),n}(\bar{\Xi}_{m,n;k}(f)_1)] \oplus f_2$. 
\item For each car $c_3$ with preference $p_3$ in $f_3$, $c_3$ will have new preference $n+1-p_3$. This yields a $\mathrm{len}(f_3)$-tuple $\nu_3$ of preferences. Then define $\bar{\Xi}_{m,n;k}(f)_3 = \bar{\Phi}_{\mathrm{len}(f_1)+\mathrm{len}(f_2),n+k;k}(\bar{\Xi}_{m,n;k}(f)_2)$ $\oplus \nu_3$. 
\item Continue this process as follows.
\begin{enumerate}
\item Suppose $i$ is odd. For each car $c_i$ with preference $p_i$ in $f_i$, $c_i$ will have new preference $n+1-p_i$. This yields a $\mathrm{len}(f_i)$-tuple $\nu_i$ of preferences, and we define \[\bar{\Xi}_{m,n;k}(f)_i = \bar{\Phi}_{\sum_{u=1}^{i-1}{\mathrm{len(f_u)}},n+k;k}(\bar{\Xi}_{m,n;k}(f)_{i-1}) \oplus \nu_i.\]
\item Suppose $i$ is even. We define \[\bar{\Xi}_{m,n;k}(f)_i = \bar{\Phi}_{\sum_{u=1}^{i-1}{\mathrm{len(f_u)}},n+k;k}(\bar{\Xi}_{m,n;k}(f)_{i-1}) \oplus f_i.\]
\end{enumerate}
\item Finally, define $\bar{\Xi}_{m,n;k}(f)$ as follows.
\begin{enumerate}
\item If $d$ is even, then $\bar{\Xi}_{m,n;k}(f) := \bar{\Xi}_{m,n;k}(f)_d$.
\item If $d$ is odd, then $\bar{\Xi}_{m,n;k}(f) := \bar{\Phi}_{m,n+k;k}(\bar{\Xi}_{m,n;k}(f)_d)$.
\end{enumerate}
\end{enumerate}
\begin{rem}
\label{obstrem}
Roughly speaking, $\bar{\Xi}_{m,n;k}(f)$ is constructed in stages $\bar{\Xi}_{m,n;k}(f)_1$, $\bar{\Xi}_{m,n;k}(f)_2$, $\ldots$, $\bar{\Xi}_{m,n;k}(f)_d$, $\bar{\Xi}_{m,n;k}(f)$, where the $\bar{\Xi}_{m,n;k}(f)_i$ are as follows:
\begin{enumerate}
\item The obstruction components of $\bar{\Xi}_{m,n;k}(f)_1$ are the reflection of the Naples components of $f_1$.
\item The obstruction components of $\bar{\Xi}_{m,n;k}(f)_2$ coincide with the Naples components of $f_1 \oplus f_2$, the difference being the lengthening by $k$ obstructions due to $\iota_k$.
\item In general, the obstruction components of $\bar{\Xi}_{m,n;k}(f)_i$
\begin{enumerate}
\item coincide with the Naples components of $f_1 \oplus f_2 \oplus \ldots \oplus f_i$ for $i$ even
\item coincide with the reflection of the Naples components of $f_1 \oplus f_2 \oplus \ldots \oplus f_i$ for $i \geq 3$ odd
\end{enumerate}
\item Since $f \notin B(m,n;k)$, the obstruction component containing $\bar{k}$ will end up on the right for odd $d$, in which case we need to apply one more reflection, so that $\bar{\Xi}_{m,n;k}(f) := \bar{\Phi}_{m,n+k;k}(\bar{\Xi}_{m,n;k}(f)_d)$.
\end{enumerate}
\end{rem}

\begin{exm}
Consider the following directed path $I$. \\
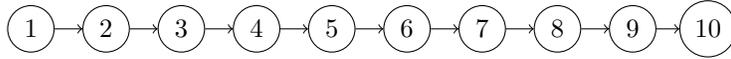
\begin{tikzpicture}
\node[shape=circle,draw=black] (1) at (0,0) {1};
\node[shape=circle,draw=black] (2) at (1,0) {2};
\node[shape=circle,draw=black] (3) at (2,0) {3};
\node[shape=circle,draw=black] (4) at (3,0) {4};
\node[shape=circle,draw=black] (5) at (4,0) {5};
\node[shape=circle,draw=black] (6) at (5,0) {6};
\node[shape=circle,draw=black] (7) at (6,0) {7};
\node[shape=circle,draw=black] (8) at (7,0) {8};
\node[shape=circle,draw=black] (9) at (8,0) {9};
\node[shape=circle,draw=black] (10) at (9,0) {10};

\path [->] (1) edge (2);
\path [->] (2) edge (3);
\path [->] (3) edge (4);
\path [->] (4) edge (5);
\path [->] (5) edge (6);
\path [->] (6) edge (7);
\path [->] (7) edge (8);
\path [->] (8) edge (9);
\path [->] (9) edge (10);
\end{tikzpicture}

For $f = (4,4,7,1,1,9,10,10,1) \in PF(9,10;4)$, we have the following: 
\begin{enumerate}
\item $\bar{\Xi}_{9,10;4}(f)_1 = (7,7,4,10)$

\item $\bar{\Xi}_{9,10;4}(f)_2 = (7,7,11,5,1)$ on \\
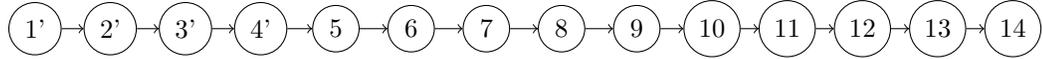
\begin{tikzpicture}
\node[shape=circle,draw=black] (1) at (0,0) {1'};
\node[shape=circle,draw=black] (2) at (1,0) {2'};
\node[shape=circle,draw=black] (3) at (2,0) {3'};
\node[shape=circle,draw=black] (4) at (3,0) {4'};
\node[shape=circle,draw=black] (5) at (4,0) {5};
\node[shape=circle,draw=black] (6) at (5,0) {6};
\node[shape=circle,draw=black] (7) at (6,0) {7};
\node[shape=circle,draw=black] (8) at (7,0) {8};
\node[shape=circle,draw=black] (9) at (8,0) {9};
\node[shape=circle,draw=black] (10) at (9,0) {10};
\node[shape=circle,draw=black] (11) at (10,0) {11};
\node[shape=circle,draw=black] (12) at (11,0) {12};
\node[shape=circle,draw=black] (13) at (12,0) {13};
\node[shape=circle,draw=black] (14) at (13,0) {14};

\path [->] (1) edge (2);
\path [->] (2) edge (3);
\path [->] (3) edge (4);
\path [->] (4) edge (5);
\path [->] (5) edge (6);
\path [->] (6) edge (7);
\path [->] (7) edge (8);
\path [->] (8) edge (9);
\path [->] (9) edge (10);
\path [->] (10) edge (11);
\path [->] (11) edge (12);
\path [->] (12) edge (13);
\path [->] (13) edge (14);
\end{tikzpicture}
\item $\bar{\Xi}_{9,10;4}(f)_3 = (7,7,4,13,9,2,1,1)$ on \\
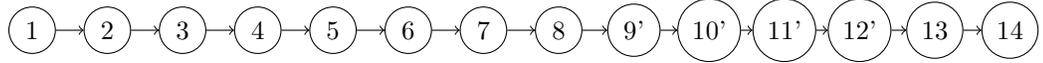
\begin{tikzpicture}
\node[shape=circle,draw=black] (1) at (0,0) {1};
\node[shape=circle,draw=black] (2) at (1,0) {2};
\node[shape=circle,draw=black] (3) at (2,0) {3};
\node[shape=circle,draw=black] (4) at (3,0) {4};
\node[shape=circle,draw=black] (5) at (4,0) {5};
\node[shape=circle,draw=black] (6) at (5,0) {6};
\node[shape=circle,draw=black] (7) at (6,0) {7};
\node[shape=circle,draw=black] (8) at (7,0) {8};
\node[shape=circle,draw=black] (9) at (8,0) {9'};
\node[shape=circle,draw=black] (10) at (9,0) {10'};
\node[shape=circle,draw=black] (11) at (10,0) {11'};
\node[shape=circle,draw=black] (12) at (11,0) {12'};
\node[shape=circle,draw=black] (13) at (12,0) {13};
\node[shape=circle,draw=black] (14) at (13,0) {14};

\path [->] (1) edge (2);
\path [->] (2) edge (3);
\path [->] (3) edge (4);
\path [->] (4) edge (5);
\path [->] (5) edge (6);
\path [->] (6) edge (7);
\path [->] (7) edge (8);
\path [->] (8) edge (9);
\path [->] (9) edge (10);
\path [->] (10) edge (11);
\path [->] (11) edge (12);
\path [->] (12) edge (13);
\path [->] (13) edge (14);
\end{tikzpicture}
\item $\bar{\Xi}_{9,10;4}(f) = \bar{\Xi}_{9,10;4}(f)_4 = (7,7,11,5,1,13,12,12,1)$ on \\
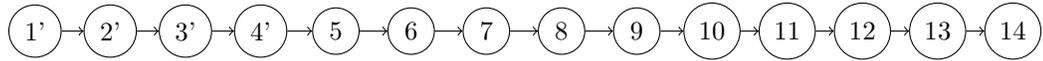
\begin{tikzpicture}
\node[shape=circle,draw=black] (1) at (0,0) {1'};
\node[shape=circle,draw=black] (2) at (1,0) {2'};
\node[shape=circle,draw=black] (3) at (2,0) {3'};
\node[shape=circle,draw=black] (4) at (3,0) {4'};
\node[shape=circle,draw=black] (5) at (4,0) {5};
\node[shape=circle,draw=black] (6) at (5,0) {6};
\node[shape=circle,draw=black] (7) at (6,0) {7};
\node[shape=circle,draw=black] (8) at (7,0) {8};
\node[shape=circle,draw=black] (9) at (8,0) {9};
\node[shape=circle,draw=black] (10) at (9,0) {10};
\node[shape=circle,draw=black] (11) at (10,0) {11};
\node[shape=circle,draw=black] (12) at (11,0) {12};
\node[shape=circle,draw=black] (13) at (12,0) {13};
\node[shape=circle,draw=black] (14) at (13,0) {14};

\path [->] (1) edge (2);
\path [->] (2) edge (3);
\path [->] (3) edge (4);
\path [->] (4) edge (5);
\path [->] (5) edge (6);
\path [->] (6) edge (7);
\path [->] (7) edge (8);
\path [->] (8) edge (9);
\path [->] (9) edge (10);
\path [->] (10) edge (11);
\path [->] (11) edge (12);
\path [->] (12) edge (13);
\path [->] (13) edge (14);
\end{tikzpicture}
\end{enumerate}
\end{exm}

\begin{exm} 
Consider the following directed path $I$. \\
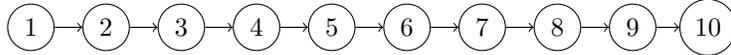
\begin{tikzpicture}
\node[shape=circle,draw=black] (1) at (0,0) {1};
\node[shape=circle,draw=black] (2) at (1,0) {2};
\node[shape=circle,draw=black] (3) at (2,0) {3};
\node[shape=circle,draw=black] (4) at (3,0) {4};
\node[shape=circle,draw=black] (5) at (4,0) {5};
\node[shape=circle,draw=black] (6) at (5,0) {6};
\node[shape=circle,draw=black] (7) at (6,0) {7};
\node[shape=circle,draw=black] (8) at (7,0) {8};
\node[shape=circle,draw=black] (9) at (8,0) {9};
\node[shape=circle,draw=black] (10) at (9,0) {10};

\path [->] (1) edge (2);
\path [->] (2) edge (3);
\path [->] (3) edge (4);
\path [->] (4) edge (5);
\path [->] (5) edge (6);
\path [->] (6) edge (7);
\path [->] (7) edge (8);
\path [->] (8) edge (9);
\path [->] (9) edge (10);
\end{tikzpicture}

For $f = (4,4,7,1,2,2,5,9,3,10) \in PF(10,10;4)$, we have the following:
\begin{enumerate}
\item $\bar{\Xi}_{10,10;4}(f)_1 = (7,7,4,10,9)$
\item $\bar{\Xi}_{10,10;4}(f)_2 = (7,7,11,5,6,2,5)$ on \\
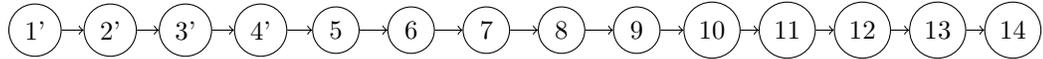
\begin{tikzpicture}
\node[shape=circle,draw=black] (1) at (0,0) {1'};
\node[shape=circle,draw=black] (2) at (1,0) {2'};
\node[shape=circle,draw=black] (3) at (2,0) {3'};
\node[shape=circle,draw=black] (4) at (3,0) {4'};
\node[shape=circle,draw=black] (5) at (4,0) {5};
\node[shape=circle,draw=black] (6) at (5,0) {6};
\node[shape=circle,draw=black] (7) at (6,0) {7};
\node[shape=circle,draw=black] (8) at (7,0) {8};
\node[shape=circle,draw=black] (9) at (8,0) {9};
\node[shape=circle,draw=black] (10) at (9,0) {10};
\node[shape=circle,draw=black] (11) at (10,0) {11};
\node[shape=circle,draw=black] (12) at (11,0) {12};
\node[shape=circle,draw=black] (13) at (12,0) {13};
\node[shape=circle,draw=black] (14) at (13,0) {14};

\path [->] (1) edge (2);
\path [->] (2) edge (3);
\path [->] (3) edge (4);
\path [->] (4) edge (5);
\path [->] (5) edge (6);
\path [->] (6) edge (7);
\path [->] (7) edge (8);
\path [->] (8) edge (9);
\path [->] (9) edge (10);
\path [->] (10) edge (11);
\path [->] (11) edge (12);
\path [->] (12) edge (13);
\path [->] (13) edge (14);
\end{tikzpicture}
\item $\bar{\Xi}_{10,10;4}(f)_3 = (11,11,4,9,10,6,9,2)$ on \\
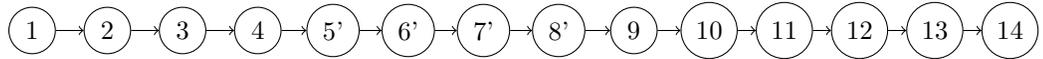
\begin{tikzpicture}
\node[shape=circle,draw=black] (1) at (0,0) {1};
\node[shape=circle,draw=black] (2) at (1,0) {2};
\node[shape=circle,draw=black] (3) at (2,0) {3};
\node[shape=circle,draw=black] (4) at (3,0) {4};
\node[shape=circle,draw=black] (5) at (4,0) {5'};
\node[shape=circle,draw=black] (6) at (5,0) {6'};
\node[shape=circle,draw=black] (7) at (6,0) {7'};
\node[shape=circle,draw=black] (8) at (7,0) {8'};
\node[shape=circle,draw=black] (9) at (8,0) {9};
\node[shape=circle,draw=black] (10) at (9,0) {10};
\node[shape=circle,draw=black] (11) at (10,0) {11};
\node[shape=circle,draw=black] (12) at (11,0) {12};
\node[shape=circle,draw=black] (13) at (12,0) {13};
\node[shape=circle,draw=black] (14) at (13,0) {14};

\path [->] (1) edge (2);
\path [->] (2) edge (3);
\path [->] (3) edge (4);
\path [->] (4) edge (5);
\path [->] (5) edge (6);
\path [->] (6) edge (7);
\path [->] (7) edge (8);
\path [->] (8) edge (9);
\path [->] (9) edge (10);
\path [->] (10) edge (11);
\path [->] (11) edge (12);
\path [->] (12) edge (13);
\path [->] (13) edge (14);
\end{tikzpicture}
\item $\bar{\Xi}_{10,10;4}(f)_4 = (7,7,11,5,7,2,5,13,3)$ on \\
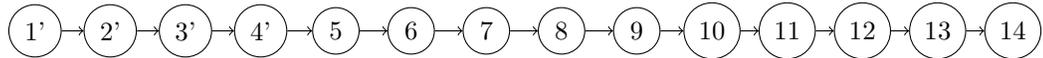
\begin{tikzpicture}
\node[shape=circle,draw=black] (1) at (0,0) {1'};
\node[shape=circle,draw=black] (2) at (1,0) {2'};
\node[shape=circle,draw=black] (3) at (2,0) {3'};
\node[shape=circle,draw=black] (4) at (3,0) {4'};
\node[shape=circle,draw=black] (5) at (4,0) {5};
\node[shape=circle,draw=black] (6) at (5,0) {6};
\node[shape=circle,draw=black] (7) at (6,0) {7};
\node[shape=circle,draw=black] (8) at (7,0) {8};
\node[shape=circle,draw=black] (9) at (8,0) {9};
\node[shape=circle,draw=black] (10) at (9,0) {10};
\node[shape=circle,draw=black] (11) at (10,0) {11};
\node[shape=circle,draw=black] (12) at (11,0) {12};
\node[shape=circle,draw=black] (13) at (12,0) {13};
\node[shape=circle,draw=black] (14) at (13,0) {14};

\path [->] (1) edge (2);
\path [->] (2) edge (3);
\path [->] (3) edge (4);
\path [->] (4) edge (5);
\path [->] (5) edge (6);
\path [->] (6) edge (7);
\path [->] (7) edge (8);
\path [->] (8) edge (9);
\path [->] (9) edge (10);
\path [->] (10) edge (11);
\path [->] (11) edge (12);
\path [->] (12) edge (13);
\path [->] (13) edge (14);
\end{tikzpicture}
\item $\bar{\Xi}_{10,10;4}(f)_5 = (9,9,13,7,8,4,7,2,5,1)$ on \\
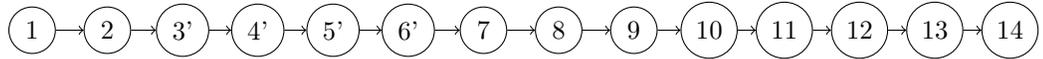
\begin{tikzpicture}
\node[shape=circle,draw=black] (1) at (0,0) {1};
\node[shape=circle,draw=black] (2) at (1,0) {2};
\node[shape=circle,draw=black] (3) at (2,0) {3'};
\node[shape=circle,draw=black] (4) at (3,0) {4'};
\node[shape=circle,draw=black] (5) at (4,0) {5'};
\node[shape=circle,draw=black] (6) at (5,0) {6'};
\node[shape=circle,draw=black] (7) at (6,0) {7};
\node[shape=circle,draw=black] (8) at (7,0) {8};
\node[shape=circle,draw=black] (9) at (8,0) {9};
\node[shape=circle,draw=black] (10) at (9,0) {10};
\node[shape=circle,draw=black] (11) at (10,0) {11};
\node[shape=circle,draw=black] (12) at (11,0) {12};
\node[shape=circle,draw=black] (13) at (12,0) {13};
\node[shape=circle,draw=black] (14) at (13,0) {14};

\path [->] (1) edge (2);
\path [->] (2) edge (3);
\path [->] (3) edge (4);
\path [->] (4) edge (5);
\path [->] (5) edge (6);
\path [->] (6) edge (7);
\path [->] (7) edge (8);
\path [->] (8) edge (9);
\path [->] (9) edge (10);
\path [->] (10) edge (11);
\path [->] (11) edge (12);
\path [->] (12) edge (13);
\path [->] (13) edge (14);
\end{tikzpicture}
\item $\bar{\Xi}_{10,10;4}(f) = \bar{\Phi}_{10,14;4}(\bar{\Xi}_{10,10;4}(f)_5) = (7,7,11,5,6,2,5,13,3,10)$ on \\
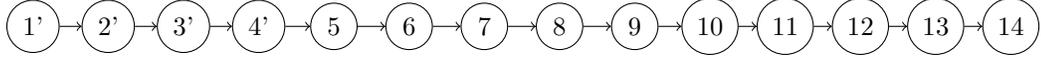
\begin{tikzpicture}
\node[shape=circle,draw=black] (1) at (0,0) {1'};
\node[shape=circle,draw=black] (2) at (1,0) {2'};
\node[shape=circle,draw=black] (3) at (2,0) {3'};
\node[shape=circle,draw=black] (4) at (3,0) {4'};
\node[shape=circle,draw=black] (5) at (4,0) {5};
\node[shape=circle,draw=black] (6) at (5,0) {6};
\node[shape=circle,draw=black] (7) at (6,0) {7};
\node[shape=circle,draw=black] (8) at (7,0) {8};
\node[shape=circle,draw=black] (9) at (8,0) {9};
\node[shape=circle,draw=black] (10) at (9,0) {10};
\node[shape=circle,draw=black] (11) at (10,0) {11};
\node[shape=circle,draw=black] (12) at (11,0) {12};
\node[shape=circle,draw=black] (13) at (12,0) {13};
\node[shape=circle,draw=black] (14) at (13,0) {14};

\path [->] (1) edge (2);
\path [->] (2) edge (3);
\path [->] (3) edge (4);
\path [->] (4) edge (5);
\path [->] (5) edge (6);
\path [->] (6) edge (7);
\path [->] (7) edge (8);
\path [->] (8) edge (9);
\path [->] (9) edge (10);
\path [->] (10) edge (11);
\path [->] (11) edge (12);
\path [->] (12) edge (13);
\path [->] (13) edge (14);
\end{tikzpicture}
\end{enumerate}
\end{exm}

\begin{thm} 
\label{obstr inj}
$\bar{\Xi}_{m,n;k}: PF(m,n;k) \rightarrow LPF(m,n+k;k)$ is an injection. 
\end{thm}
\begin{proof}
For $g_1 \in B(m,n;k)$ and $g_2 \in PF(m,n;k)-B(m,n;k)$, $\bar{\Xi}_{m,n;k}(g_2)$ has cars preferring vertices in $[k]$ while $\bar{\Xi}_{m,n;k}(g_1)$ does not. Hence $\bar{\Xi}_{m,n;k}[B(m,n;k)]$ and $\bar{\Xi}_{m,n;k}[PF(m,n;k)-B(m,n;k)]$ do not intersect. We know that $\Xi_{m,n;k}$ is a bijection between $B(m,n;k)$ and $PF(m,n)$, and $\iota_k$ naturally identifies $PF(m,n)$ as the subset of $LPF(m,n+k;k)$ where cars do not prefer the obstructed vertices. It remains to prove the claim for the case $\bar{\Xi}_{m,n;k} \mid PF(m,n;k)-B(m,n;k)$.

Let $f \in PF(m,n;k)-B(m,n;k)$ with $k$-decomposition $f = f_1 \oplus f_2 \oplus \cdots \oplus f_d$. Notice that the obstruction components of $\bar{\Xi}_{m,n;k}(f)_1$ are the reflection of those of $f \mid [\mathrm{len}(f_1)]$. For $\bar{\Xi}_{m,n;k}(f)_2 = \iota_k[\Phi_{\mathrm{len}(f_1),n}(\bar{\Xi}_{m,n;k}(f)_1)] \oplus f_2$, notice that $\Phi_{\mathrm{len}(f_1),n}(\bar{\Xi}_{m,n;k}(f)_1)$ has exactly the same obstruction components as $f \mid [\mathrm{len}(f_1)]$ and that $\iota_k$ merely ``adds $k$ obstructions to the left'', so the cars of $f_2$ park in $\bar{\Xi}_{m,n;k}(f)_2$ exactly as they do in $f$.

In general, it is easy to observe the following: 
\begin{enumerate}
\item For $i$ even, the obstructions components of $\bar{\Phi}_{\sum_{u=1}^{i-1}{\mathrm{len(f_u)}},n+k;k}(\bar{\Xi}_{m,n;k}(f)_{i-1})$ coincide with the Naples components of $f$ restricted to $[\sum_{u=1}^{i-1}{\mathrm{len(f_u)}}]$ (shifted forward $k$ spots), with the exception of the first obstruction component which is the first Naples component of $f \mid [\sum_{u=1}^{i-1}{\mathrm{len(f_u)}}]$ lengthened by the $k$ obstructions. Hence the cars of $f_i$ park in $\bar{\Xi}_{m,n;k}(f)_i$ exactly as they do in $f$.
\item For $i \geq 3$ odd, the obstruction components of $\bar{\Phi}_{\sum_{u=1}^{i-1}{\mathrm{len(f_u)}},n+k;k}(\bar{\Xi}_{m,n;k}(f)_{i-1})$ are the reflection of the Naples components of $f$ restricted to $[\sum_{u=1}^{i-1}{\mathrm{len(f_u)}}]$, with the exception of the last obstruction component which is the reflection of the first Naples component of $f \mid [\sum_{u=1}^{i-1}{\mathrm{len(f_u)}}]$ lengthened by the $k$ obstructions. Hence the cars of $f_i$ have traverse paths in $\bar{\Xi}_{m,n;k}(f)_i$ that are the reflection of those in $f$.
\end{enumerate}

If $d$ is even, then $\bar{\Xi}_{m,n;k}(f) = \bar{\Xi}_{m,n;k}(f)_d$ has all its obstructions on the left. If $d$ is odd, then $\bar{\Xi}_{m,n;k}(f) := \bar{\Phi}_{m,n+k;k}(\bar{\Xi}_{m,n;k}(f)_d)$ ensures that all its obstructions are on the left. These two cases clearly do not intersect, since car $m$ in one traverses more than $k$ vertices while car $m$ in the other does not. In both cases, the obstruction components of $\bar{\Xi}_{m,n;k}(f)$ coincide with the Naples components of $f$, excepting the $k$ obstructions themselves. Since $f$ is a parking function, $\bar{\Xi}_{m,n;k}(f)$ thus defined is an element of $LPF(m,n+k;k)$.

To prove the injectivity of $\bar{\Xi}_{m,n;k}$, we first introduce some notations. Since car preferences are included sequentially in the procedure defining $\bar{\Xi}_{m,n;k}(f)$, let $\bar{\Xi}_{m,n;k}(f)^c$ denote the stage in the construction of $\bar{\Xi}_{m,n;k}(f)$ after the preference of car $c$ has been included. If $c \in [\sum_{u=1}^{i}{\mathrm{len}(f_u)}] - [\sum_{u=1}^{i-1}{\mathrm{len}(f_u)}]$, then $\bar{\Xi}_{m,n;k}(f)^c$ is a $c$-tuple that is a sub-tuple of $\bar{\Xi}_{m,n;k}(f)_i$.

For any car $a \leq c \in [m]$, we denote by $\mathrm{comp}^c(a,f)$ the obstruction component of $\bar{\Xi}_{m,n;k}(f)^c$ where $a$ parks. We also denote by $\mathrm{comp}(a,f)$ the obstruction component of $\bar{\Xi}_{m,n;k}(f)$ where car $a$ parks. If $p_a$, $p_b$ are the preferences of cars $a$, $b$ respectively in $\bar{\Xi}_{m,n;k}(f)^c$, let $\mathrm{diff}^c(a,b;f) = p_b-p_a$ which is their difference in $\bar{\Xi}_{m,n;k}(f)^c$. If $p'_a$, $p'_b$ are the preferences of cars $a$, $b$ respectively in $\bar{\Xi}_{m,n;k}(f)$, let $\mathrm{diff}(a,b;f) = p'_b-p'_a$. We denote by $\mathrm{tpath}^c(a,\bar{\Xi})$ the path car $a$ must traverse before parking in $\bar{\Xi}_{m,n;k}(f)^c$, and by $\mathrm{tpath}(a,\bar{\Xi})$ the traverse path of $a$ in $\bar{\Xi}_{m,n;k}(f)$.

Suppose that $p_a$, $p_b$ are both in the same obstruction component in $\bar{\Xi}_{m,n;k}(f)^c$, then we have $\mathrm{diff}^c(a,b;f) = \mathrm{diff}^{d'}(a,b;f)$ for any $d' > c$, and $\mathrm{diff}^c(a,b;f) = \mathrm{diff}(a,b;f)$. In other words, the difference between any two preferences in the same obstruction component is preserved by the procedure defining $\bar{\Xi}_{m,n;k}(f)$.

It is also clear that $\mathrm{tpath}^c(a,\bar{\Xi})$, $\mathrm{tpath}^{d'}(a,\bar{\Xi})$, and $\mathrm{tpath}(a,\bar{\Xi})$ all have the same length, for any $d' > c$.  In other words, the procedure defining $\bar{\Xi}_{m,n;k}(f)$ preserves the traverse path length of each car.

Assume for a contradiction that $f \neq g \in PF(m,n;k)-B(m,n;k)$ but $\bar{\Xi}_{m,n;k}(f) = \bar{\Xi}_{m,n;k}(g)$. Let $f = f_1 \oplus f_2 \oplus \ldots \oplus f_{d_1}$ and $g = g_1 \oplus g_2 \oplus \ldots \oplus g_{d_2}$ be their respective $k$-decompositions. Let $j \in [m]$ be maximal such that car $j$ has different preferences in $f$ and $g$; label these preferences as $p_j$ and $q_j$ respectively, with $f_{c_1}$ containing $p_j$ and $g_{c_2}$ containing $q_j$.

The above observations force $d_1 = d_2$, $c_1 = c_2$, and $\mathrm{len}(f_l) = \mathrm{len}(g_l)$ for all $l \in [d_1]$. Let $p'_j$ and $q'_j$ denote the preferences of car $j$ in $\bar{\Xi}_{m,n;k}(f)_{c_1}$ and $\bar{\Xi}_{m,n;k}(g)_{c_1}$ respectively. By assumption, car $l$ has the same preference in both $f$ and $g$, for any $l > j$. We show that the obstruction components of car $j$ will ultimately differ in $\bar{\Xi}_{m,n;k}(f)$ and $\bar{\Xi}_{m,n;k}(g)$.

Car $j$ will have the same traverse path length $t_j \geq 0$ in both $\bar{\Xi}_{m,n;k}(f)_{c_1}$ and $\bar{\Xi}_{m,n;k}(g)_{c_1}$. For every car $c \in [j]$, $c$ is in $\mathrm{comp}^j(j,g)$ if and only if $c$ is in $\mathrm{comp}^j(j,f)$, and we have $\mathrm{diff}^j(j,c;g) = \mathrm{diff}^j(j,c;f)$; otherwise, 
the cars in $[j]$ would yield different obstruction components in $\bar{\Xi}_{m,n;k}(f)\mid[j]$ and $\bar{\Xi}_{m,n;k}(g)\mid[j]$, which is a contradiction. Hence the car preference placements are the same in $\mathrm{comp}^j(j,g)$ and $\mathrm{comp}^j(j,f)$, though these two components occupy different sets of vertices of $I$ due to divergent preferences of car $j$.

Then the procedure will reflect $\mathrm{comp}^j(j,f)$ and $\mathrm{comp}^j(j,g)$ a number of times; an even number of reflections will restore the original positions of $\mathrm{comp}^j(j,f)$ and $\mathrm{comp}^j(j,g)$ provided no new cars will be added. If the procedure adds no new cars to $\mathrm{comp}^j(j,f)$ and $\mathrm{comp}^j(j,g)$, then $\mathrm{comp}(j,g)$ and $\mathrm{comp}(j,f)$ would still occupy different sets of vertices of $I$ in $\bar{\Xi}_{m,n;k}(g)$ and $\bar{\Xi}_{m,n;k}(f)$ respectively. If the traverse path of car $j' > j$ intersects $\mathrm{comp}^j(j,f)$ or $\mathrm{comp}^j(j,g)$, then the cars in $[j']$ would yield different obstruction components in $\bar{\Xi}_{m,n;k}(f)\mid[j']$ and $\bar{\Xi}_{m,n;k}(g)\mid[j']$, because $j'$ has the same preference in both situations by assumption.

Therefore, in all cases we have $\bar{\Xi}_{m,n;k}(f) \neq \bar{\Xi}_{m,n;k}(g)$, which is the desired contradiction. This shows that $\bar{\Xi}_{m,n;k}$ is indeed injective.
\end{proof}
\begin{cor}
\label{gen bound}
We have $$|PF(m,n;k)| \leq |LPF(m,n+k;k)|,$$ with equality only if $k=0$.
\end{cor}
\begin{proof}
The bound $|PF(m,n;k)| \leq |LPF(m,n+k;k)|$ follows immediately from Theorem \ref{obstr inj}. For any $k \geq 1$, no element of $LPF(m,n+k;k)$ whose car $1$ prefers an obstructed vertex will be in $\bar{\Xi}_{m,n;k}(PF(m,n;k))$, so $|PF(m,n;k)| < |LPF(m,n+k;k)|$.
\end{proof}


Lastly, consider the case $m = n$, where we denote $LPF_{n+k,k} := LPF(n,n+k;k)$ and $PF_{n,k} := PF(n,n;k)$. The main result (Theorem 1.2) of \cite{Ehren} yields the following special case
\begin{prop}
\label{ktrailer}
We have \[|LPF_{n+k,k}| 
= (k+1)(k+n+1)^{n-1}. \]
\end{prop}

Corollary \ref{gen bound} and Proposition \ref{ktrailer} imply the following 
\begin{thm} 
\label{concr bound}
We have \[|PF_{n,k}| \leq (k+1)(k+n+1)^{n-1}, \] with equality only if $k=0$.
\end{thm}

\end{document}